\documentclass[12pt,a4paper,twoside,reqno]{amsart}
\title[Limits of balanced metrics]{Limits of balanced metrics on vector bundles and polarised manifolds}
\author[M. Garcia-Fernandez]{Mario Garcia-Fernandez}
  \address{Centre for Quantum Geometry of Moduli Spaces\\ Aarhus Universitet\\
    Ny Munkegade 118\\ DK-8000 Aarhus C, Denmark}
  \email{mariogf@imf.au.dk}
\author[J. Ross]{Julius Ross}
  \address{Department of Pure Mathematics and Mathematical Statistics, University of Cambridge, Wilberforce Road,Cambridge, CB3 0WB, UK}
  \email{j.ross@dpmms.cam.ac.uk}

\thanks{The first author is supported by QGM (Centre for Quantum Geometry of Moduli Spaces) funded by the Danish National Research Foundation. The second author is partially supported by a Marie Curie Grant (PIRG-GA-2008-230920). Part of this work was undertaken during the first author's visit to the Isaac Newton Institute.}

\usepackage{hyperref,url}
\usepackage{amssymb,latexsym}
\usepackage{amsmath,amsthm}
\usepackage[latin1]{inputenc}
\usepackage{enumerate}
\usepackage[all]{xy} \CompileMatrices
\SelectTips{cm}{12} 
\theoremstyle{plain}
\newtheorem{theorem}{Theorem}[section]
\newtheorem{lemma}[theorem]{Lemma}

\newtheorem{proposition}[theorem]{Proposition}

\theoremstyle{definition}
\newtheorem{definition}[theorem]{Definition}
\newtheorem{definition-theorem}[theorem]{Definition-Theorem}

\theoremstyle{remark}
\newtheorem{remark}[theorem]{Remark}

\newcommand{\secref}[1]{\S\ref{#1}}

\numberwithin{equation}{section} 

\newcommand{\tr}{\operatorname{tr}}
\newcommand{\Id}{\operatorname{Id}}

\newcommand{\End}{\operatorname{End}}

\newcommand{\dbar}{\bar{\partial}}

\newcommand{\CC}{{\mathbb C}}
\newcommand{\PP}{{\mathbb P}}

\newcommand{\RR}{{\mathbb R}}
\newcommand{\ZZ}{{\mathbb Z}}
\newcommand{\rk}{\operatorname{rk}}

\renewcommand{\(}{\left(}
\renewcommand{\)}{\right)}

\newcommand{\defeq}{\mathrel{\mathop:}=} 

\newcommand{\surj}{\to\kern-1.8ex\to}
\newcommand{\hra}{\hookrightarrow}

\newcommand{\cA}{\mathcal{A}}

\newcommand{\cJ}{\mathcal{J}}

\newcommand{\cP}{\mathcal{P}}

\newcommand{\cG}{\mathcal{G}}
\newcommand{\cL}{\mathcal{L}}

\newcommand{\cS}{\mathcal{S}}
\newcommand{\cT}{{\mathcal{T}}}

\newcommand{\Lie}{\operatorname{Lie}}

\newcommand{\LieG}{\operatorname{Lie} \cG}
\newcommand{\cX}{{\widetilde{\mathcal{G}}}}

\newcommand{\cH}{\mathcal{H}} 
\newcommand{\LieH}{\Lie\cH}

\newcommand{\cHL}{\cH_L}
\newcommand{\cXL}{\cX_L}

\newcommand{\LieXL}{\operatorname{Lie} \cX_L}
\begin{document}

\maketitle

\begin{abstract}
We consider a notion of balanced metrics for triples $(X,L,E)$ which depend on a parameter $\alpha$, where $X$ is smooth complex manifold with an ample line bundle $L$ and $E$ is a holomorphic vector bundle over $X$. For generic choice of $\alpha$, we prove that the limit of a convergent sequence of balanced metrics leads to a Hermitian-Einstein metric on $E$ and a constant scalar curvature K\"ahler metric in $c_1(L)$. For special values of $\alpha$, limits of balanced metrics are solutions of a system of coupled equations relating a Hermitian-Einstein metric on $E$ and a K\"ahler metric in $c_1(L)$. For this, we compute the top two terms of the density of states expansion of the Bergman kernel of $E \otimes L^k$.
\end{abstract}


\section{Introduction}

\begin{flushright}
\emph{}

\end{flushright}

There are two well known theories in which the existence of a canonical metric in K\"ahler geometry is related to a stability condition in algebraic geometry.  One is the Hitchin-Kobayashi correspondence for Hermitian-Einstein metrics on vector bundles, and the other the Yau-Tian-Donaldson conjecture for constant scalar curvature K\"ahler (cscK) metrics on projective manifolds.

In each of these contexts there is a crucial role played by balanced metrics.  On the one hand, the existence of a balanced metric can be shown to be equivalent to a stability condition in the sense of finite dimensional Geometric Invariant Theory \cite{Luo,W2,Zh}.  On the other hand, the asymptotic behaviour of a sequence of balanced metrics is governed by a ``density of states'' expansion, from which the Hermitian-Einstein or cscK equations can be extracted.

In this paper we combine these ideas by considering simultaneously stability of  a vector bundle and its underling manifold.  Let $(X,L,E)$ be a triple consisting of a smooth complex manifold $X$ of dimension $n$ with an ample holomorphic line bundle $L$ and a holomorphic vector bundle $E$ of rank $r$.  Picking a basis for the space of sections of $L^{\otimes k}$ and for $E\otimes L^{\otimes k}$ where $k\gg 0$ gives an embedding
\begin{equation}
 \phi\colon X \hra \mathbb P \times \mathbb G\label{eq:embeddingpg}
\end{equation}
into a product of a projective space and a Grassmannian of $r$-planes.    This yields a stability problem (of the associated point in the Hilbert scheme of $\mathbb P\times \mathbb G$) under the natural group action given by the freedom in the choice of bases.  As usual, stability depends on a choice of linearisation, and below we consider a natural class of such choices that depends on a parameter $\alpha\in \mathbb R^{n+2}$.  One expects that such stability to be connected to some kind of canonical metrics, and we show this to be the case.  In fact
\begin{enumerate}
\item For generic choice of $\alpha$, asymptotic stability (as $k$ tends to infinity) is related to asking for both a  Hermitian-Einstein metric on $E$ and a cscK metric in $c_1(L)$.   Thus, from an algebraic view, asymptotic stability should be the simultaneous asymptotic stability of $(X,L)$ as a polarised manifold and slope stability of $E$ as a vector bundle.\vspace{2mm}
\item For particular choices of $\alpha$, asymptotic stability is related to a certain ``coupled equation'' that asks for a Hermitian-Einstein metric on $E$ and a K\"ahler metric in $c_1(L)$ whose scalar curvature is related to second order terms derived from the curvature of the metric on $E$.
\end{enumerate}\medskip

It should be emphasised that for a genuine stability condition one would additionally require that $\alpha$ is chosen so that the corresponding linearisation (or symplectic form) is positive, otherwise stability should be taken in a formal sense.   This will not be relevant for the purposes of this paper since we consider the balanced metrics directly (although it is natural to expect that their existence is equivalent to a stability just as \cite{Luo,W2,Zh}).\medskip


To state a precise theorem, suppose that $h$ is a hermitian metric on $L$ with curvature form $-i\omega$, so the induced metric $h^k$ on $L^{\otimes k}$ has curvature $-i\omega_1:= -ik\omega$.  Similarly if $H$ is a hermitian metric on the bundle $E$ with curvature form $F_H$ then the induced hermitian metric on $\det (E\otimes L^{\otimes k})$ has curvature form $-i\omega_2:= -irk\omega + \tr F_H$. Now, given $\alpha= (\alpha_0,\ldots,\alpha_{n+1})\in \mathbb R^{n+2}$ define an $L^2$-inner product on the space $H^0(L^{\otimes k})$ using the fibre metric $h^k$ and the volume form $dV_1/\int_XdV_1$, where
\begin{eqnarray*}
  dV_1 = \sum_{p=1}^{n+1} p \alpha_{p} \omega_1^{p-1} \wedge \omega_2^{n+1-p}.
\end{eqnarray*}
Similarly we define an $L^2$-metric on $H^0(E\otimes L^{\otimes k})$ from the fibre metric $H\otimes h^k$ but this time using the volume form $dV_2/\int_XdV_2$, where
\begin{equation}
  dV_2 = \sum_{p=0}^{n} (n+1-p) \alpha_p \omega_1^p \wedge \omega_2^{n-p}.
\end{equation}
(These forms arise naturally from the moment map theory; see \secref{subsec:balancedemb}). The non-vanishing of $dV_j$ at any point of $X$ will be justified asymptotically, as $k$ tends to infinity, under natural assumptions on the parameter $\alpha$ \eqref{eq:betasnonzero}.\medskip


We say that a pair $(h,H)$ is \emph{$\alpha$-balanced} (with respect to $k$) if there is a choice of orthonormal bases such that the induced Fubini-Study metrics pull back under  $\phi\colon X\hookrightarrow \mathbb P\times \mathbb G$ to give $(h^k,H\otimes h^k)$.

\begin{theorem}
Fix $\alpha$ and suppose $(h_k,H_k)$ is a sequence of hermitian metrics on $L$ and $E$ respectively which converges (in $C^{\infty}$ say) to $(h,H)$ as $k$ tends to infinity. Suppose furthermore that $(h_k,H_k)$ is $\alpha$-balanced with respect to $k$ for $k$ sufficiently large. Then
\begin{enumerate}
\item For general choice of $\alpha$, the limit satisfies the equations
  \begin{equation}
i \Lambda F_H = \lambda \Id, \qquad S_\omega = \hat S\label{eq:uncoupledintro}
\end{equation}
where $S_{\omega}$ denotes the scalar curvature of $\omega$, $\Lambda$ denotes the trace with respect to $\omega$ and $\lambda$, $\hat S$ are topological constants.
\item For special choice of $\alpha$,  the limit  satisfies the equations
\begin{equation}\label{eq:coupledalpha1intro}
i\(\Lambda F_H - \frac{1}{r}\Lambda\tr F_H\Id\) + \(\frac{1}{2}S_{\omega} - \frac{\beta_1}{\beta_0}i\Lambda \tr F_H\)\Id = \lambda' \Id,
\end{equation}
\begin{equation}\label{eq:coupledalpha2intro}
(\Delta - 4\lambda')i\Lambda \tr F_H - \tr\Lambda^2(F_H^2 + F_H \wedge \tr R_\omega) - \kappa\Lambda^2 (\tr F_H)^2 = c,
\end{equation}
where $i\tr R_\omega$ is the Ricci-curvature of $\omega$ and $\lambda',\beta_i,\kappa,c$ are topological quantities that depend on $\alpha$ \eqref{eq:definitionoftopologicalconstants}.
\end{enumerate}

\end{theorem}

Motivated by the work of Luo \cite{Luo}, Wang \cite{W1} and Donaldson \cite{D1},  it is natural to conjecture that the existence of solutions of \eqref{eq:uncoupledintro} implies the existence of balanced metrics and the asymptotic stability of the triple $(X,L,E)$, at least for general $\alpha$ which give rise to a positive linearisation.   The confirmation of this conjecture would lead to a physical interpretation of the GIT moduli space of semistable triples when $c_1(X) =0$, $c_1(E) = 0$, $c_2(X) = c_2(E)$, as in this case a solution of \eqref{eq:uncoupledintro} corresponds to a compactification of Heterotic String Theory in the large volume limit \cite{St}. Constructions of moduli spaces of triples using GIT have been carried out by Gieseker-Morrison \cite{GM} and Pandharipande \cite{Pd} for the case of curves, but apparently there is no such construction for higher dimensional $X$. It is worth to point out that, for smooth curves, GIT stability of a triple in the sense of~\cite{GM,Pd} is equivalent to the stability of the vector bundle over the curve, as expected from Theorem 1.1 (1) (see \cite[Theorem~1.1, Proposition~8.2.1]{Pd}).

We expect the analysis of the equivalent question for the system given by \eqref{eq:coupledalpha1intro} and \eqref{eq:coupledalpha2intro} to be more delicate. The first equation in (2) is an endomorphism valued equation that can be reduced to the Hermitian-Einstein equation after a conformal change; the second is a scalar equation relating the metric on the base to a second order quantity in the curvature of the metric on the bundle. These equations are coupled when $\beta_1 \neq 0$ and of a similar form to those studied by the first author jointly with Luis Alvarez-Consul and Oscar Garcia-Prada \cite{AGG,GF}, although they are not precisely the same.  In fact, one of the main differences between these and \cite{AGG,GF} is that the equations here are invariant under rescaling of the metric for each fixed $\alpha \in\RR^{n+2}$ and depend strongly on the topologies of $X$, $L$ and $E$. We hope to address the analysis of the set of $\alpha$'s which lead to positive linearisations and produce coupled equations in a future work. \medskip

Our method of proof is based on that of Donaldson \cite[Theorem~2]{D1} which is outlined at the start of \secref{sec:ProofTheorem}.  This in turn relies on the density of states expansion of the Bergman kernel \cite{C1,F1,MM1,T1,Y1,Z1}.  For our purpose we need to know the top two terms of this expansion in the case of twisting by a vector bundle.
\begin{theorem}\label{th:2}
  Let $\{s_j\}$ be a basis for $H^0(E\otimes L^{\otimes k})$ that is orthonormal with respect to the fibre metric $H\otimes h^k$ and the fixed volume form $\omega^n/n!$.  Then the Bergman kernel $B_k= \sum_{i} s_j s_j^*$ has a $C^\infty$ asymptotic expansion
$$ (2\pi)^nB_k = B_0 k^n + B_1 k^{n-1} + B_2 k^{n-2} + O(k^{n-3})$$
where $B_i$ are smooth functions on $X$ depending on $H$ and $h$. In particular,
$$ B_0 = Id, \quad B_1 = i\Lambda F_H + \frac{S_{\omega}}{2}Id\quad \text{and}$$
\begin{equation*}
  \begin{split}
    B_2 &=  -\frac{1}{2}\Delta_{\dbar}(i\Lambda F_H) - \frac{1}{2}\Lambda F_H \Lambda F_H - \frac{1}{2}
(F_H)_{j,\overline{k}}(F_H)_{k,\overline{j}}\\
& + \frac{1}{2}S_\omega i \Lambda F_H - \frac{1}{2}(F_H)_{j,\overline{k}}\tr(R_\omega)_{k,\overline{j}}\\
& - \frac{1}{6}\Delta(S_\omega)\Id + \frac{1}{24}\(|R_\omega|^2 - 4|\tr R_\omega|^2 + 3 S_\omega^2\)\Id.
  \end{split}
\end{equation*}
\end{theorem}

\begin{remark}
After this work was completed we noticed that the $B_2$ term has been calculated independently by L. Wang \cite{WangLi} and H. Xu \cite{Xu} although using different methods (see Remark \ref{rem:notationdiff} about the different notations).
\end{remark}

The important part of this theorem for our purpose is the explicit equation for the $B_2$ term, which follows from the recursion method of Berman-Berndtsson-Sj\"{o}strand \cite{BeBoSJ} after an elementary, but laborious, calculation (see Lu \cite{Lu} for an alternative approach in the non-twisted case).  As far as we can see, the contributions from $B_2$ is the only way that one can relate stability to equations that are second order in the curvature of the bundle, and since this appears naturally among coupled equations we expect this result to be of independent interest.\medskip

\noindent{\bf Outline:} We start in \secref{sec:prelim} with preliminaries and the definition of $\alpha$-balanced metrics.  This is put into an infinite dimensional gauge-theoretic context in \secref{sec:infinite}.  The computation of the Bergman Kernel is contained in \secref{sec:Bergman}, and the proof of the theorem on limits of balanced metrics in \secref{sec:ProofTheorem}. \medskip

\noindent {\bf Acknowledgements:} We wish to thank Luis Alvarez-Consul, Oscar Garcia-Prada, Julien Keller and Richard Thomas for useful discussions.

\section{Balanced conditions for bundles and polarised manifolds}\label{sec:prelim}

\subsection{Preliminaries}
\label{subsec:Finite-dim1}
In this section we study Hamiltonian actions on the space of embeddings of a smooth manifold into a product of symplectic manifolds. When the source manifold is complex and the target manifolds are K\"ahler, we will endow the space of holomorphic embeddings with a family of closed $(1,1)$-forms that will be used in \secref{subsec:balancedemb} to define balanced conditions.

Let $X$ be a $2n$-dimensional real smooth compact manifold, and $(S_1,\omega_1)$ and $(S_2,\omega_2)$ be two symplectic manifolds of dimensions $m_1$ and $m_2$, with $m_1 + m_2 > 2n$. Consider the product $S = S_1 \times S_2$ and let
$$
\cS \subset C^\infty(X,S).
$$
be the space of smooth embeddings $\phi\colon X \to S$. Given integers $p,q$, define a closed $2q$-form on $S$ by
\begin{equation}\label{equation:sigmabalanced}
\sigma_{p,q} = \omega_1^p \wedge\omega_2^{q-p},
\end{equation}
(so by convention $\sigma_{p,q}=0$ if $p<0$ or $p>q$).  Here and in the sequel, we omit pull-backs to simplify the notation when there is no possible confusion.
Now fix a non-negative integer $0 \leq p \leq n+1$. Then, $\sigma_{p,n+1}$ induces a $2$-form $\Omega_p$ on $\cS$ defined by
\begin{equation}\label{eq:Omegasigma}
\Omega_p(V_1,V_2) = \int_X V_2\lrcorner(V_1 \lrcorner \sigma_{p,n+1}),
\end{equation}
where $V_j \in T_\phi\cS \cong \Gamma(X,\phi^*TS)$. A straightforward computation shows that
$$
d\Omega_p(V_1,V_2,V_3) = \int_X V_3\lrcorner(V_2 \lrcorner(V_1 \lrcorner d\sigma_{p,n+1})) = 0,
$$
and therefore $\Omega_p$ is closed.

Let $\cH_j$ be the group of Hamiltonian symplectomorphisms of $(S_j,\omega_j)$. Then the group
$$
\cH = \cH_1 \times \cH_2
$$
acts on $\cS$ on the right by composition, and this action preserves $\Omega_p$.  We note that $\Omega_p$ may not be positive definite, and thus not define a symplectic form, but nevertheless seek to find a moment map for the $\cH$-action on $(\cS,\Omega_p)$; that is an equivariant map
\begin{equation}\label{eq:mmap}
\mu_p\colon \cS \to \LieH^*
\end{equation}
which satisfies
\begin{equation}\label{eq:mmapcondition}
d \langle \mu,f\rangle = Y_f \lrcorner \Omega_p
\end{equation}
for any $f =(f_1,f_2) \in \LieH$, where $Y_f$ denotes the infinitesimal action of $f$ on $\cS$.  To ease the computations we identify $\LieH_j$ with $C^\infty_0(S_j)$, the smooth functions with zero integral with respect to $\omega_j^{m_j}$.
\begin{lemma}\label{lem:mup}
There exists a moment map for the $\cH$-action on $(\cS,\Omega_p)$, given by
\begin{equation}\label{eq:Omegasigmammap}
\langle\mu_p(\phi),f\rangle = p\int_X \phi^* f_1 \sigma_{p-1,n} + (n+1-p)\int_X
\phi^* f_2 \sigma_{p,n}.
\end{equation}
\end{lemma}
\begin{proof}
Note first that $Y_f$ is given by the pull-back $\phi^*y_f \in \Gamma(X,\phi^*TS)$, where $y_f$ denotes the Hamiltonian vector field of $f_1 + f_2$ on $(S,\omega_1 + \omega_2)$. Then, given an arbitrary $V \in \Gamma(X,\phi^*TS)$ we have that $\Omega_p(Y_f,V)$ equals
\begin{equation*}
\begin{split}
p&\int_X \phi^*\(V \lrcorner \((y_f \lrcorner \omega_1) \wedge \sigma_{p-1,n}\)\) + (n+1-p)\int_X \phi^* \(V \lrcorner \((y_f \lrcorner \omega_2) \wedge \sigma_{p,n}\)\)\\
& = p\int_X \phi^*\(V \lrcorner d\(f_1 \sigma_{p-1,n}\)\)+ (n+1-p)\int_X \phi^* \(V \lrcorner d\(f_2\sigma_{p,n}\)\)\\
& = d\langle\mu_p,f\rangle(V).
\end{split}
\end{equation*}
as required.
\end{proof}
Suppose now that $X$ is a complex manifold and that $\omega_1$ and $\omega_2$ are K\"ahler. Then, the product complex structure on $S$ induces a complex structure on the space of holomorphic embeddings $\cS_h$. We can use the forms $\Omega_p$ to define a family of closed $(1,1)$-forms on $\cS_h$. For this, expanding $\Omega_p$ as above and using that $\omega_1$ and $\omega_2$ are of type $(1,1)$ one can easily see that \eqref{eq:Omegasigma} defines a $(1,1)$-form on $\cS_h$. Hence, given a parameter $\alpha = (\alpha_0,\ldots,\alpha_{n+1})\in \RR^{n+2}$ we define the family of closed $(1,1)$-forms by
\begin{equation}\label{eq:Omegaalpha}
\Omega_\alpha = \sum_{p=0}^{n+1}\alpha_p\Omega_p.
\end{equation}
As an immediate consequence of Lemma~\ref{lem:mup} we obtain the following.
\begin{proposition}\label{propo:mualpha}
The $\cH$-action on $(\cS,\Omega_\alpha)$ is Hamiltonian with moment map
\begin{equation}\label{eq:mualpha}
\mu_\alpha = \sum_{p=0}^{n+1}\alpha_p\mu_p.
\end{equation}
\end{proposition}

The form $\Omega_\alpha$ is the curvature of a natural line bundle on $\cS_h$ that can be constructed by the Deligne Pairing.   Suppose that $\omega_i$ is the curvature of a hermitian metric on a line bundle $L_i$ over $S_i$.   Then, following \cite{PhSt}, let
\begin{equation}\label{eq:universalSh}
\mathcal U= \{(x,\phi) \in S\times \cS_h \colon x \in \phi(X)\}
\end{equation}
be the universal family.    Let $\pi\colon \mathcal U\to S$ the projection onto the first factor and consider the line bundle on $\cS_h$ given by the Deligne pairing
$$
\cL_p = \langle \underbrace{\pi^*L_1, \ldots, \pi^*L_1}_{p\, \mbox{\tiny times}}, \underbrace{\pi^*L_2,\ldots, \pi^*L_2}_{n+1-p\,\, \mbox{\tiny times }}\rangle.
$$
This pairing was introduced in \cite{D}, but see also \cite{PhSt} or \cite{Zh} for details.  Then, on the smooth locus of $\cS_h$,  there is a natural hermitian metric on $\cL_p$ with curvature given by
$$
\int_{\mathcal U/\cS_h}\omega_1^p\wedge\omega_2^{n+1-p} = \Omega_p
$$
and so $\Omega_\alpha$ is the curvature of the induced hermitian metric on the line bundle
$$
\cL_\alpha = \cL_0^{\alpha_0} \otimes \ldots \otimes \cL_{n+1}^{\alpha_{n+1}}.
$$

\begin{remark}
The set of $\alpha$ for which \eqref{eq:Omegaalpha} defines a K\"ahler form is non-empty. To see this, let $\gamma_1,\gamma_2 \in \RR$ be positive real numbers. Then, the K\"ahler form
$$
\omega_\gamma = \gamma_1 \omega_1 + \gamma_2 \omega_2
$$
on $S$ induces a $2$-form $\Omega_\gamma$ on $\cS$, by replacing $\sigma_{p,n+1}$ with $\omega_\gamma^{n+1}$ in~\eqref{eq:Omegasigma}, which can be expressed as
$$
\Omega_\gamma = \Omega_\alpha
$$
for a suitable choice of $\alpha$. Now, the positivity of $\Omega_\gamma$ follows from the identity
$$
\Omega_\gamma(V,JV) = (n+1)\int_X\omega_\gamma(V,JV)\omega_\gamma^n + (n+1)n\int_X|V \lrcorner\omega_\gamma|^2\omega_\gamma^n \geq 0,
$$
where $J$ denotes the product complex structure on $S$ and $|V \lrcorner\omega_\gamma|^2$ is the norm of the pull-back of $V \lrcorner\omega_\gamma$ to $X$ with respect to the K\"ahler form induced by $\omega_\gamma$.
\end{remark}


\subsection{Balanced embeddings, bundles and polarised manifolds}
\label{subsec:balancedemb}
We next use the family of moment maps considered above to define a notion of $\alpha$-balanced triple $(X,L,E)$.  This notion generalises the notion of balanced polarised manifolds, defined first by Luo~\cite{Luo} and adapted to vector bundles by Wang~\cite{W1,W2}.
Let $X$ be a smooth $n$-dimensional complex manifold. Given positive integers $r,N,M$ we define
$$
S = \PP(\CC^N) \times \mathbb G(r,\CC^M),
$$
where $\PP(\CC^N)$ is the complex projective space of dimension $N-1$ endowed with the Fubini-Study metric $\omega_{P}$ and $\mathbb G(r,\CC^M)$ is the Grassmannian of $r$ dimensional quotients of $\CC^M$ endowed with the Fubini-Study metric $\omega_G$.
Using the same notation as in~\secref{subsec:Finite-dim1}, consider the closed $(1,1)$-form $\Omega_\alpha$ on $\cS_h$ given by \eqref{eq:Omegaalpha}. Since we have a homomorphism
$$
SU(N) \times SU(M) \to \cH_1 \times \cH_2
$$
from the product of the special unitary groups,  the moment map~\eqref{eq:mualpha} induces a moment map for the $SU(N) \times SU(M)$-action on $\cS_h$.  The $\alpha$-balanced embeddings are defined to be the zeros of this moment map
$$
\mu_\alpha\colon \cS_h \to \mathfrak{su}(N)^*\times \mathfrak{su}(M)^*.
$$
\begin{definition}\label{def:balancedemb}
A holomorphic embedding $\phi\colon X \to M$ is $\alpha$-balanced if
$$
\langle\mu_\alpha(\phi),\zeta\rangle = 0,
$$
for all $\zeta \in \mathfrak{su}(N)\times \mathfrak{su}(M)$.
\end{definition}
Taking standard homogeneous coordinates $[Z_1:\ldots:Z_N]$ on $\PP(\CC^N)$ and $A$ on $\mathbb G(r,\CC^M)$, where $A$ denotes an $M \times r$ matrix representing a point on $\mathbb G(r,\CC^M)$, the balanced condition for $\phi \in \cS_h$ is equivalent to (cf.~\cite{Luo,W4})

\begin{equation}\label{eq:balancedcondition}
\begin{split}
\int_X \phi^*\(\frac{ZZ^*}{|Z|^2}\sigma_1\) & = \frac{\int_X \phi^*\sigma_1}{N}\Id,\\
\int_X \phi^*\(A(A^*A)^{-1}A^*\sigma_2\) &= \frac{\int_X \phi^*\sigma_2}{M}\Id,
\end{split}
\end{equation}
where
\begin{equation}
\begin{split}
\sigma_1 & = \sum_{p=1}^{n+1}p\alpha_p\omega_P^{p-1}\wedge\omega_G^{n+1-p},\\
\sigma_2 & = \sum_{p=0}^{n}(n+1-p)\alpha_p\omega_P^{p}\wedge\omega_G^{n-p}.
\end{split}
\end{equation}
Note that $A(A^*A)^{-1}A^*$ measures the pointwise hermitian product of the hyperplane sections on $\mathbb G(r,\CC^M)$ with respect to the Fubini-Study metric (similarly for $ZZ^*/|Z|^2$). \medskip


Now let $L$ be an ample line bundle and $E$ a holomorphic vector bundle over $X$.  For $k \gg 0$ set
\begin{equation}\label{eq:NkMk}
N_k = \dim H^0(L^{\otimes k}) \quad \text{ and } \quad M_k = \dim H^0(E \otimes L^{\otimes k}).
\end{equation}
Any choice of basis $\{t_j\}$ for $H^0(L^{\otimes k})$ for $k\gg 0$ gives an embedding $X\subset \PP(H^0(L^{\otimes k})) \simeq \mathbb P(\CC^{N_k})$.  Analogously, given a basis $\{s_j\}$ for $H^0(E\otimes L^{\otimes k})$ the surjection
$$ \mathbb C^{M_k} \simeq H^0(E\otimes L^{\otimes k}) \twoheadrightarrow E$$
gives an embedding $X\subset \mathbb G(r,\mathbb C^{M_k})$.

\begin{definition}
The triple $(X,L,E)$ is said to be \emph{$\alpha$-balanced with respect to $k$} if there exists basis $\{t_j\}$ of $H^0(L^{\otimes k})$ and $\{s_j\}$ of $H^0(E \otimes L^{\otimes k})$ such that the corresponding embedding in $\PP(\CC^{N_k}) \times \mathbb G(r,\CC^{M_k})$ is $\alpha$-balanced.
\end{definition}
We next characterise the $\alpha$-balanced condition in terms of the existence of balanced metrics. To state this we recall notation from the introduction: let $h,H$ be hermitian metrics on $L$ and $E$ respectively, with curvature forms $-i\omega$ and $F_{H}$.  Set $\omega_{1} = k\omega$ and $\omega_{2} = kr\omega + i\tr F_{H}$ and consider volume forms on $X$ given by
\begin{eqnarray*}
dV_1 &=&  \sum_{p=1}^{n+1} p \alpha_{p+1} \omega_1^{p-1} \wedge \omega_2^{n+1-p},\\
dV_2  &=& \sum_{p=0}^{n} (n+1-p) \alpha_p \omega_1^p \wedge \omega_2^{n-p}.
\end{eqnarray*}
By integrating over $X$ we get an $L^2$ metric on $H^0(L^{\otimes k})$ given by the fibre metric $h^k$ and the volume form $dV_1/\int_XdV_1$. Now, given any orthonormal basis $\{t_j\}$ we define the Bergman function by
\begin{equation}\label{eq:bergmanfunction1}
\rho_k(h,H) = \sum_{j} |t_j|^2
\end{equation}
where the norm on the right is taken fibrewise. Similarly there is an $L^2$ metric on $H^0(E\otimes L^{\otimes k})$ induced by the fibre metric $H\otimes h^k$ and the volume form $dV_2/\int_XdV_2$ and given an orthonormal basis $\{s_j\}$ an associated (endomorphism valued) Bergman function
\begin{equation}\label{eq:bergmanfunction2}
B_k(h,H) = \sum_{j} s_j s_j^{*},
\end{equation}
where the dual is taken with respect to the fibre metric $H\otimes h^k$.  The non-vanishing of $dV_j$ at any point of $X$ will be justified in \secref{sec:ProofTheorem} asymptotically in $k$ under natural assumptions on the parameter $\alpha$, see \eqref{eq:betasnonzero}. The normalisation of the volume forms will be used to compute asymptotic expansions for the relevant Bergman functions (see \eqref{eq:varphijpositive} and Lemma \ref{lem:expansion}).
\begin{proposition}\label{prop:balancedmetrics}
The triple $(X,L,E)$ is $\alpha$-balanced with respect to $k$  if and only if there exists metrics $h$ on $L$ and $H$ on $E$ and orthonormal basis such that the associated Bergman kernels are both constant over $X$; i.e.\
\begin{equation}\label{eq:bergmanconstant}
  \rho_k(h,H) = N_k, \qquad B_k(h,H) = \frac{M_k}{r} \Id.
\end{equation}
\end{proposition}
\begin{proof}
For the only if part, denote $\phi = \phi_1\times\phi_2$ and define
\begin{align*}
h &= \phi_1^*h_{P}^{1/k},\\
H &= \phi_2^*H_{G} \otimes h_k^{-l},
\end{align*}
where $h_P$ and $H_G$ denote the Fubini-Study metrics on the universal quotient bundles of $\PP(\CC^{N_k})$ and $\mathbb G(r,\CC^{M_k})$ respectively. Then, by the $\alpha$-balanced condition the pullbacks $\{t_j\}$ and $\{s_j\}$ of the hyperplane sections are orthonormal with respect to the relevant $L^2$-metrics, since $\phi^* \omega_P = k\omega$ and $\phi^*\omega_G = i\tr F_{H_G} = (kr \omega + i\tr F_{H})$, and it is clear that satisfy~\eqref{eq:bergmanconstant} (up to rescaling). The converse is proved similarly.
\end{proof}
\begin{remark}
  Using the above one sees that the balanced condition can be recast in terms of pullback of the appropriate Fubini-Study metric, as described in the introduction.
\end{remark}

\section{Bergman Kernel asymptotics}\label{sec:Bergman}
In this section we calculate higher order terms of the asymptotic expansion of the Bergman Kernel for $H^0(X,E\otimes L^{\otimes k})$ using the iterative formula in~\cite{BeBoSJ}. Recall that the Bergman Kernel is the Kernel of integration of the orthogonal projection
$$
\Pi\colon \Gamma(E\otimes L^{\otimes k})_{L^2} \to H^0(E\otimes L^{\otimes k})
$$
from the space of $L^2$ sections to the space of holomorphic sections, with respect to the volume form $\omega^n/n!$ determined by
\begin{equation}\label{eq:omega}
\omega = iF_h
\end{equation}
and the metric $H \otimes h^{\otimes k}$. It can identified with a section of the vector bundle $(E\otimes L^{\otimes k})^* \boxtimes E\otimes L^{\otimes k}$ over $X \times X$ given by
$$
B_k(x,y) = \sum_j s_j(x)(\cdot,s_j(y))_{H \otimes h^{\otimes k}},
$$
where $\{s_j\}$ is an orthonormal basis of $H^0(X,E \otimes L^{\otimes k})$, and admits an asymptotic expansion of the form
\begin{equation}\label{eq:Bexpansionxy}
(2\pi)^nB_k(x,y) = \Id k^n + B_1(x,y) k^{n-1} + B_2(x,y) k^{n-2} + \ldots
\end{equation}
as $k$ goes to infinity. The precise result we need (see e.g. \cite[Theorem 4.1.1]{MM1}) concerns the previous expansion over the diagonal, where $B_k$ induces an endomorphism of $E \otimes L^{\otimes k}$ given by
\begin{equation}\label{eq:Bexpansionxx}
B_k(x) = B_k(x,x) = \sum_j s_j(x)(\cdot,s_j(x))_{H \otimes h^{\otimes k}}.
\end{equation}
\begin{theorem}\label{th:Bergman}
The Bergman kernel has a $C^\infty$ asymptotic expansion over the diagonal
$$
(2\pi)^nB_k = B_0 k^n + B_1 k^{n-1} + B_2 k^{n-2} + O(k^{n-3})
$$
where $B_0 = \Id$ and $B_i$ are endomorphism valued smooth functions on $X$ depending on $H$ and $h$. More precisely, for any $l,N \geq 0$ there exists a constant $K(l,N,H,h)$ such that
$$
\Bigg{\|}(2\pi)^nB_k - \sum_{j=0}^N B_ik^{n-j}\Bigg{\|}_{C^l} \leq K(l,N,H,h)k^{n-N-1}.
$$
Moreover, the expansion is uniform in that for any $l,N$ there is an integer $m$ such that if $(H,h)$ runs over a set of metrics which are bounded in $C^m$, and with $(H,h)$ bounded below, the constants $K(l,N,H,h)$ are bounded by a constant $K(l,N)$ independent of $(H,h)$.
\end{theorem}

To calculate the first coefficients of the asymptotic expansion of $B_k$ we use the \emph{local asymptotic Bergman Kernels} constructed in~\cite{BeBoSJ}. To state the result that we will use (\cite[Theorem 3.1]{BeBoSJ}), fix $x_0 \in X$, local holomorphic coordinates $x \in U \subset \CC^r$ centred at $x_0$ and trivialisations of $L$ and $E$ such that
$$
(u_1,u_2)_h = u_1 \overline{u_2} e^{-\phi}, \qquad (s_1,s_2)_H = \overline{s_2}^T H s_1
$$
where $\phi$ is a smooth function on $\CC^n$ and $H$ a hermitian matrix-valued smooth function of rank $r$, the rank of $E$, and we use column notation for the local sections $s_j$ of $E$. By~\eqref{eq:omega}, we have
$$
\omega = i\partial\dbar\phi = ig_{j,\overline k} dx^j \wedge d\overline{x}^k,
$$
where we sum over repeated indices. To simplify some of the formulae, this summation convention will be assumed in the sequel when there is no possible confusion. We denote by $g$ the matrix $(g_{j,\overline k})$. Let
$$
\Lambda\colon \Omega^{p,q} \to \Omega^{p-1,q-1}
$$
be the contraction operator on forms. Denoting $(g^{j,\overline k})^T = g^{-1}$, $\Lambda$ acts on $(1,1)$ forms by
$$
\Lambda\(a_{j,\overline{k}} dx^j \wedge d\overline{x}^k\) = -i g^{j,\overline k}a_{j,\overline k}.
$$

Let $\psi(x,y)$ and $G(x,y)$ be almost holomorphic extensions of $\phi$ and $H$ on
$$
\overline{\Delta} = \{(x,y): y = \overline{x}\} \subset U \times U.
$$
By definition, $\psi\colon U\times U \to \CC$ is smooth, satisfies $\psi(x,\overline{x}) = \phi(x)$ and
\begin{equation}\label{eq:almosthol}
D^\alpha(\dbar \psi)_{|\overline{\Delta}} = 0
\end{equation}
for all multiindex $\alpha$, and similarly for $G$. Let $\theta\colon U \times U \times U \to \CC^n$ be the smooth function of of $3n$ variables defined by
\begin{equation}\label{eq:thetaj}
\theta^j(x,y,z) = \int_0^1 \partial_j \psi(tx + (1-t)y,z)dt,
\end{equation}
where $\partial_j\psi=\partial_j\psi(x,y)$ denotes partial holomorphic differentiation with respect to $x_j$. Using the identity
\begin{equation}\label{eq:theta_z}
\theta(y,y,z) = \partial_y \psi(y,z),
\end{equation}
it follows that the map
\begin{equation}\label{eq:coordinatechange}
(x,y,z) \to (x,y,\theta)
\end{equation}
defines a smooth change of coordinates around the origin in $\CC^{3n}$, as $\theta(0) = \partial_x\phi(0)=0$ and
$$
d_z\theta(0) = \partial_z \theta(0) = \partial_z\partial_y \psi(0) = \dbar_x\partial_x\phi (0) = \Id,
$$
where $d_z = \partial_z + \dbar_z$ and the first equality follows from \eqref{eq:almosthol}. Define functions
\begin{align*}
\Delta_0(x,y,\theta) & = \det \partial_y\partial_z \psi(y,z)/\det \partial_z \theta(x,y,z),\\
\Delta_G'(x,y,\theta) & = G(x,z)^{-1}G(y,z)
\end{align*}
where $z = z(x,y,\theta)$, and let
$$
\Delta_G(x,y,\theta) = \Delta_0 \Delta_G'.
$$

\begin{theorem}(Berman-Berndtsson-Sj\"{o}strand)
The coefficients in the expansion~\eqref{eq:Bexpansionxx} satisfy the following recursive formula for $m> 0$
\begin{equation}\label{eq:Recursive}
\sum_{l=0}^m\frac{(D_\theta \cdot D_y)^l}{l!}(B_{m-l}\Delta_G)_{|x=y} = 0,
\end{equation}
$D_\theta \cdot D_y = \sum_{j=1}^n \partial_{\theta^j}\partial_{y^j}$, $B_k = B_k(x,z)$, $x$ is considered fixed and $z = z(y,\theta)$.
\end{theorem}
We use now \eqref{eq:Recursive} to calculate $B_1$ and $B_2$ in~\eqref{eq:Bexpansionxx}, given by
\begin{equation}\label{eq:b1}
B_1(x,z(x,x,\theta)) = - D_\theta \cdot D_y (\Delta_G)_{|x=y}
\end{equation}
\begin{equation}\label{eq:b2}
B_2(x,z(x,x,\theta)) = \(- D_\theta \cdot D_y(B_1\Delta_G) - \frac{1}{2}(D_\theta \cdot D_y)^2(\Delta_G)\)_{|x=y}.
\end{equation}

For the calculations we define $\Psi = \Psi(y,z)$ as the $n \times n$ matrix of partial derivatives $\partial_z\partial_y\psi(y,z)$, which satisfies $\Psi(x,\overline{x}) = g^T(x)$. Considering $G$ as a function on $(y,z)$ variables, we have
\begin{equation}\label{eq:etaR}
\Psi^{-1}\partial_y \Psi(x,\overline{x}) = \eta(x), \qquad \partial_z(\Psi^{-1}\partial_y \Psi)(x,\overline{x}) = R(x),
\end{equation}
\begin{equation}\label{eq:ThetaF}
G^{-1}\partial_y G(x,\overline{x}) = \Theta(x), \qquad \partial_z(G^{-1}\partial_yG)(x,\overline{x}) = F(x),
\end{equation}
where $\eta$ and $R = \dbar \eta$ denote the Chern connection of $\omega$ and its curvature and $\Theta$ and $F = \dbar \Theta$ the Chern connection of $H$ and its curvature. Without loss of generality, we assume that
\begin{equation}\label{eq:metricasympt}
e^{-\phi} = 1 + O(|x|^2), \qquad g(x) = \Id + O(|x|^2), \qquad H(x) = \Id + O(|x|^2),
\end{equation}
so that $\eta(0) = (g^T)^{-1}\partial g^T(0) = 0$ and $\Theta(0) = H^{-1}\partial H(0)=0$. Further, we will use the following properties of the change of coordinates \eqref{eq:coordinatechange}, which follow from elementary calculations using \eqref{eq:metricasympt}.
\begin{lemma}\label{lem:propertiesz}

\begin{align}
\partial_\theta z & = (\partial_z\theta)^{-1}\nonumber\\
\partial_y z (0) &= 0\\
\partial^2_{\theta,y}z(0) &= \partial^2_{\theta,\theta}z(0) = 0\nonumber
\end{align}

\end{lemma}

For simplicity we will assume that $\phi$ and $H$ are real analytic and so $\psi$, $\Psi$, $G$ and $\theta$ are holomorphic. The general case follows easily combining the computations below with condition \eqref{eq:almosthol}. Hence, the following Taylor expansions in $y$ around $x=y$ hold.
\begin{lemma}\label{lem:bergman0}
\begin{equation}\label{eq:taylorDelta0}
\begin{split}
\Delta_0 &= 1 + c_1 + c_2 + c_3 + c_4 + \ldots
\end{split}
\end{equation}
where
\begin{equation}\label{eq:Delta0ci}
\begin{split}
c_1 & = \frac{1}{2}\tr\Psi^{-1}\partial_{y^k}\Psi(x,z)(y^k-x^k)\\
c_2 & = \frac{1}{3}\tr\partial_{y^k}(\Psi^{-1}\partial_{y^m}\Psi)(x,z)(y^k-x^k)(y^m-x^m)\\
c_3 & = \frac{1}{8}\tr\Psi^{-1}(\partial_{y^k}\Psi)\tr\Psi^{-1}(\partial_{y^m}\Psi)(x,z)(y^k-x^k)(y^m-x^m)\\
c_4 & = - \frac{1}{24}\tr\Psi^{-1}(\partial_{y^k}\Psi)\Psi^{-1}(\partial_{y^m}\Psi)(x,z)(y^k-x^k)(y^m-x^m)
\end{split}
\end{equation}
and
\begin{equation}\label{eq:taylorDeltaG'}
\begin{split}
\Delta_G' &= \Id + d_1 + d_2 + d_3 + \ldots
\end{split}
\end{equation}
where
\begin{equation}\label{eq:DeltaG'di}
\begin{split}
d_1 & = (G^{-1}\partial_{y^k}G)(x,z)(y^k - x^k)\\
d_2 & = \frac{1}{2}(G^{-1}\partial_{y^k}G)(G^{-1}\partial_{y^m}G)(x,z)(y^k-x^k)(y^m-x^m)\\
d_3 & = \frac{1}{2}\partial_{y^k}(G^{-1}\partial_{y^m}G)(x,z)(y^k-x^k)(y^m-x^m)
\end{split}
\end{equation}
\end{lemma}
\begin{proof}
Note first that the Taylor expansion in $y$ around $x=y$ of the integrand in~\eqref{eq:thetaj} is
\begin{align*}
\partial_j\psi(tx+(1-t)y,z) & = \partial_j \psi(x,z) + (1-t)\partial^2_{k,j}\psi(x,z)(y^k-x^k)\\
& + (1-t)^2 \frac{1}{2}\partial^3_{m,k,j}\psi(x,z)(y^k-x^k)(y^m-x^m) + \ldots,
\end{align*}
and therefore
\begin{equation}\label{eq:partialztheta}
\begin{split}
\partial_z\theta(x,y,z) & = \Psi(x,z) + \frac{1}{2}(\partial_{y^k}\Psi(x,z))(y^k-x^k)\\
& + \frac{1}{6}(\partial_{y^k}\partial_{y^m}\Psi(x,z))(y^k-x^k)(y^m-x^m) + \ldots.
\end{split}
\end{equation}
Expanding now $\Psi^{-1}(y,z)$ we obtain
\begin{equation}\label{eq:taylorgthetaz}
\begin{split}
\Psi^{-1}\partial_z\theta(x,y,z)&= \Id - \frac{1}{2}\Psi^{-1}\partial_{y^k}\Psi(x,z)(y^k-x^k)\\
& + \frac{1}{2}\Psi^{-1}(\partial_{y^k}\Psi)\Psi^{-1}(\partial_{y^m}\Psi)(x,z)(y^k-x^k)(y^m-x^m)\\
& - \frac{1}{3}\Psi^{-1}\partial_{y^k}\partial_{y^m}\Psi(x,z)(y^k-x^k)(y^m-x^m) + \ldots
\end{split}
\end{equation}
To prove \eqref{eq:taylorDelta0}, note now that $\Delta_0 = \det(\Id + P)^{-1}$, for a suitable matrix $P$. Combining the expansions for the determinant and the inverse of a matrix of the form $\Id + P$, we find
\begin{align*}
\Delta_0 &= 1 + \frac{1}{2}\tr\Psi^{-1}\partial_{y^k}\Psi(x,z)(y^k-x^k)\\
& - \frac{1}{2}\tr\Psi^{-1}(\partial_{y^k}\Psi)\Psi^{-1}(\partial_{y^m}\Psi)(x,z)(y^k-x^k)(y^m-x^m)\\
& + \frac{1}{3}\tr\Psi^{-1}\partial_{y^k}\partial_{y^m}\Psi(x,z)(y^k-x^k)(y^m-x^m)\\
& + \frac{1}{8}\tr\Psi^{-1}(\partial_{y^k}\Psi)\tr\Psi^{-1}(\partial_{y^m}\Psi)(x,z)(y^k-x^k)(y^m-x^m)\\
& + \frac{1}{8}\tr\Psi^{-1}(\partial_{y^k}\Psi)\Psi^{-1}(\partial_{y^m}\Psi)(x,z)(y^k-x^k)(y^m-x^m) + \ldots
\end{align*}
and hence one sees that \eqref{eq:taylorDelta0} follows from
$$
\Psi^{-1}\partial_{y^m}\partial_{y^k}\Psi = \partial_{y^m}(\Psi^{-1}\partial_{y^k}\Psi) + (\Psi^{-1}\partial_{y^m}\Psi)(\Psi^{-1}\partial_{y^k}\Psi).
$$
Similarly, \eqref{eq:taylorDeltaG'} follows from a similar calculation and the Taylor expansion of $\Delta_G'$ combined with
$$
G^{-1}\partial_{y^m}\partial_{y^k}G = \partial_{y^m}(G^{-1}\partial_{y^k}G) + (G^{-1}\partial_{y^m}G)(G^{-1}\partial_{y^k}G).
$$
\end{proof}
From the previous expansions we obtain the following, where $S = \Lambda i \tr R$ is the scalar curvature of $\omega$ and we use the same notation as in \eqref{eq:etaR}, \eqref{eq:ThetaF}.
\begin{lemma}\label{lem:bergman1}
\begin{equation}\label{eq:DeltaGtheta}
\begin{split}
\partial_{\theta^j}\Delta_{0|y=x} & = 0 \\
\partial_{\theta^j}\Delta'_{G|y=x} &= 0,
\end{split}
\end{equation}
\begin{equation}\label{eq:Delta0yDeltaGRF}
\begin{split}
\partial_{y^j}\Delta_0(x,x,\overline x) & = \frac{1}{2}\tr \eta_j(x),\\
\partial_{y^j}\Delta_G'(x,x,\overline x) & = \Theta_j(x),\\
\partial_{z^k}\partial_{y^j}\Delta_0(x,x,\overline x) & = - \frac{1}{2}\tr R_{j,\overline k}(0),\\
\partial_{z^k}\partial_{y^j}\Delta_G'(x,x,\overline x) & = - F_{j,\overline k}(0),
\end{split}
\end{equation}
\begin{equation}\label{eq:SLambdaFxoverlinex}
\begin{split}
D_\theta\cdot D_y(\Delta_0)(x,x,\overline x) & = - \frac{1}{2}S(x),\\
D_\theta\cdot D_y(\Delta_G')(x,x,\overline x) & = - i\Lambda F(x).
\end{split}
\end{equation}
\end{lemma}
\begin{proof}
Formulae \eqref{eq:DeltaGtheta} follow from Lemma \ref{lem:bergman0}. Using the same Lemma, it follows that
\begin{align*}
\partial_{y^j}\Delta_0(x,x,z) & = \frac{1}{2}\tr (\Psi^{-1}\partial_{y^j}\Psi)(x,z),\\
\partial_{y^j}\Delta_G'(x,x,z) &= (G^{-1}\partial_{y^j}G)(x,z),
\end{align*}
which combined with \eqref{eq:etaR} and \eqref{eq:ThetaF} leads to \eqref{eq:Delta0yDeltaGRF}. Finally, using $\partial_\theta z = (\partial_z\theta)^{-1}$ (see Lemma \ref{lem:propertiesz}), we have that the LHS of the first equation in \eqref{eq:SLambdaFxoverlinex} equals
\begin{equation}\label{eq:Sxoverlinex2}
\begin{split}
\frac{1}{2}\partial_{\theta^j}\tr (\Psi^{-1}\partial_{y^j} \Psi)(x,\overline x) & = \frac{1}{2}\tr(\partial_{z^k}(\Psi^{-1}\partial_{y^j} \Psi))\partial_{\theta^j}z^k(x,\overline x)\\
& = \frac{1}{2}\Psi^{-1}_{j,k}\tr\partial_{z^k}(\Psi^{-1}\partial_{y^j} \Psi)(x,\overline x)\\
& = \frac{1}{2} g^{j,\overline k}\tr R_{\overline k,j}(x) = - \frac{1}{2}S(x),
\end{split}
\end{equation}
and that the LHS of the second equation in \eqref{eq:SLambdaFxoverlinex} equals
\begin{equation}\label{eq:LambdaFxoverlinex2}
\begin{split}
\partial_{\theta^j}(G^{-1}\partial_{y^j}G)(x,\overline x) &= \Psi^{-1}_{j,k}\partial_{z^k}(G^{-1}\partial_{y^j} G)(x,\overline x)\\
& = g^{j,\overline k}F_{\overline k,j}(x) = - g^{j,\overline k}F_{j,\overline k}(x) = - i\Lambda F(x).
\end{split}
\end{equation}
\end{proof}

From Lemma \ref{lem:bergman1} and \eqref{eq:b1} we obtain
\begin{align*}
B_1(x,z(x,x,\theta)) & = - D_\theta \cdot D_y(\Delta_0)_{|x=y}\Id - D_\theta \cdot D_y(\Delta_G')_{|x=y}\\
& - (\partial_{\theta^j}\Delta_0)(\partial_{y^j}(G^{-1}(x,z)G(y,z)))_{|x=y}\\
& - (\partial_{y^j}\Delta_0)(\partial_{\theta^j}(G^{-1}(x,z)G(y,z)))_{|x=y}\\
& = - D_\theta \cdot D_y(\Delta_0)_{|x=y}\Id - D_\theta \cdot D_y(\Delta_G')_{|x=y},
\end{align*}
which gives
$$
B_1(0) = \(i\Lambda F + \frac{S}{2}\Id\)(0).
$$
To compute~\eqref{eq:b2} we will use the following formulae, where $\Delta$ (without subscripts) denotes the Laplacian \begin{equation}\label{eq:Laplacian}
\Delta = 2i\Lambda \dbar\partial = -2g^{j,\overline{k}}\partial_j\dbar_k.
\end{equation}
\begin{lemma}\label{lem:bergman2}

\begin{equation}\label{eq:Delta0DeltaG'second}
\partial^2_{\theta^k,\theta^j}\Delta_0(0) = 0 \quad\text{and}\quad \partial^2_{\theta^k,\theta^j}\Delta_G'(0) = 0
\end{equation}\vspace{-4mm}
\begin{align}
D_\theta \cdot D_y(B_1)(0) &= 0\label{eq:DthetaDyB1}\\
\Psi^{-1}_{m,l}\partial_{z^l}\partial_{y^m}(\Psi^{-1}_{j,k}\tr\partial_{z^k}(\Psi^{-1}\partial_{y^j} \Psi))(0) &= \frac{1}{2}\Delta(S)(0),\label{eq:LaplacianS}\\\Psi^{-1}_{m,l}\partial_{z^l}\partial_{y^m}(\Psi^{-1}_{j,k}\partial_{z^k}(G^{-1}\partial_{y^j}G))(0) &= \frac{1}{2}\Delta (i\Lambda F)(0),\label{eq:LaplacianF}
\end{align}

\end{lemma}
\begin{proof}
Formulae \eqref{eq:Delta0DeltaG'second} follow from~\eqref{eq:taylorDelta0} and~\eqref{eq:taylorDeltaG'}. Formula \eqref{eq:DthetaDyB1} follows from Lemma \ref{lem:propertiesz} and the equality
$$
D_\theta \cdot D_y(B_1)(0) = (\partial_{z^m}\partial_{z^k}B_1)(\partial_{\theta^j}z^m)(\partial_{y^j}z^k)(0).
$$
Using ~\eqref{eq:Sxoverlinex2} it follows that the LHS of \eqref{eq:LaplacianS} equals
$$
- g^{m,\overline{l}}\dbar_l\partial_m(S)(0) = \frac{1}{2}\Delta(S)(0).
$$
Finally, it follows from \eqref{eq:LambdaFxoverlinex2} that the LHS of \eqref{eq:LaplacianF} equals
$$
- g^{m,\overline{l}}\dbar_l\partial_m(i\Lambda F)(0) = \frac{1}{2}\Delta(i\Lambda F)(0).
$$
\end{proof}
Hence, the previous lemma combined with~\eqref{eq:DeltaGtheta} and \eqref{eq:Delta0yDeltaGRF} leads to
\begin{align*}
D_\theta \cdot D_y(B_1\Delta_G)(0) & = D_\theta \cdot D_y(B_1)(0) + B_1(0) D_\theta \cdot D_y(\Delta_G)(0)\\
& + \((\partial_{\theta^j} B_{1})(\partial_{y^j} \Delta_{G}) + (\partial_{\theta^j}\Delta_{G})(\partial_{y^j}B_1)\)(0)\\
& = - B_1^2(0).
\end{align*}
To conclude, we calculate the terms in
\begin{align*}
(D_\theta \cdot D_y)^2(\Delta_G)(0) & = D_\theta\cdot D_y \Bigg{(}(D_\theta \cdot D_y \Delta_0)\Delta_G' + \Delta_0 (D_\theta \cdot D_y \Delta_G')\\
& + (\partial_{\theta^j}\Delta_0)(\partial_{y^j}(\Delta_G') + (\partial_{y^j}\Delta_0)(\partial_{\theta^j}(\Delta_G'))\Bigg{)}(0)\\
& = \Bigg{(}(D_\theta \cdot D_y)^2(\Delta_0)\Id + (D_\theta \cdot D_y)^2(\Delta_G')\\
& + 2(D_\theta \cdot D_y\Delta_0)(D_\theta \cdot D_y\Delta_G')\\
& + 2(\partial_{\theta^j}\partial_{y^k}\Delta_0)(\partial_{\theta^k}\partial_{y^j}\Delta_G')\Bigg{)}(0),
\end{align*}
where for the last equality we have used~\eqref{eq:DeltaGtheta},~\eqref{eq:Delta0yDeltaGRF} and \eqref{eq:Delta0DeltaG'second}. For simplicity of the formulae, in the sequel we omit the evaluation of the functions at $0$. From Lemma \ref{lem:bergman1} we obtain
\begin{align*}
(D_\theta \cdot D_y\Delta_0)(D_\theta \cdot D_y\Delta_G') & = \frac{1}{2}S \cdot i \Lambda F\\
(\partial_{\theta^j}\partial_{y^k}\Delta_0)(\partial_{\theta^k}\partial_{y^j}\Delta_G') & = \frac{1}{2}F_{j,\overline{k}}\tr R_{k,\overline{j}},
\end{align*}
where $i\tr R$ is the Ricci form of $\omega$. To compute $(D_\theta D_y)^2\Delta_0$ at the origin, using Lemma \ref{lem:propertiesz} and Lemma \ref{lem:bergman1} we calculate
\begin{align*}
(D_\theta \cdot D_y)^2(c_1) & = \partial^2_{\theta^j,\theta^k}(\partial^2_{y^j,y^k}c_1)\\
& = \partial^2_{\theta^j,\theta^k}(\tr(\partial_{z^l}(\Psi^{-1}\partial_{y^j}\Psi))(x,z)\partial_{y^k}z^l)\\
& = \tr(\partial_{z^l}(\Psi^{-1}\partial_{y^j}\Psi))(0)\partial^2_{\theta^k,y^k}(\partial_{\theta^j}z^l)\\
& = - (\tr R)_{j,\overline{l}}\partial^2_{\theta^k,y^k}((\partial_z\theta)^{-1}_{l,j})\\
& = \frac{1}{2}(\tr R_{j,\overline{l}})(\partial_{\theta^k}(\Psi^{-1}\partial_{y^k}\Psi(y,z)))_{l,j}\\
& = - \frac{1}{2}(\tr R_{j,\overline{l}})(R_{k,\overline{k}})_{l,j} = - \frac{1}{2}|\tr R|^2,
\end{align*}
where $z = z(y,\theta)$ and we have used the expansion \eqref{eq:partialztheta} of to compute the partial derivatives of the $(l,j)$ component of the matrix $(\partial_z\theta)^{-1}(y,z)$ for the fifth equality. Similarly, using Lemmas \ref{lem:propertiesz}, \ref{lem:bergman1} and \ref{lem:bergman2}, we obtain (after some computation that is relegated to Appendix \ref{sec:appendix})
\begin{align*}
(D_\theta \cdot D_y)^2(c_2) & = \frac{1}{3}\Delta(S) + \frac{2}{3}|\tr R|^2,\\
(D_\theta \cdot D_y)^2(c_3) & = \frac{1}{4}S^2 + \frac{1}{4}|\tr R|^2,\\
(D_\theta \cdot D_y)^2(c_4) & = - \frac{1}{12}|\tr R|^2 - \frac{1}{12}|R|^2,\\
(D_\theta \cdot D_y)^2(d_1) & = - F_{j,\overline{l}}(R_{k,\overline{k}})_{l,j},\\
(D_\theta \cdot D_y)^2(d_2) & = \frac{1}{2}\Delta(i\Lambda F) + (R_{j,\overline{j}})_{k,m}F_{m,\overline{k}},\\
(D_\theta \cdot D_y)^2(d_2) & = - \Lambda F \Lambda F + F_{j,\overline{k}}F_{k,\overline{j}}
\end{align*}
and therefore
\begin{align*}
(D_\theta \cdot D_y)^2(\Delta_0) & = \frac{1}{3}\Delta(S) - \frac{1}{12}(|R|^2 - 4|\tr R|^2 - 3 S^2),\\
(D_\theta \cdot D_y)^2(\Delta_G') & = \frac{1}{2}\Delta(i\Lambda F) - \Lambda F \Lambda F + F_{j,\overline{k}}F_{k,\overline{j}}.
\end{align*}
From the formulae above we conclude that
\begin{align*}
B_2 & = \(B_1^2 - \frac{1}{2}(D_\theta \cdot D_y)(\Delta_G)\)\\
& = -\frac{1}{2}\Delta_{\dbar}(i\Lambda F) - \frac{1}{2}\Lambda F \Lambda F - \frac{1}{2}
F_{j,\overline{k}}F_{k,\overline{j}}\\
& + \frac{1}{2}S i \Lambda F - \frac{1}{2}F_{j,\overline{k}}\tr R_{k,\overline{j}}\\
& - \frac{1}{6}\Delta(S)\Id + \frac{1}{24}\(|R|^2 - 4|\tr R|^2 + 3 S^2\)\Id,
\end{align*}
where $\Delta_{\dbar}$ is the $\dbar$ Laplacian acting on smooth endomorphisms of $E$
\begin{equation}\label{eq:LaplacianH}
\Delta_{\dbar} = -i\Lambda\partial_H\dbar,
\end{equation}
determined by $\omega$, $H$ and the holomorphic structure on the bundle (see e.g. \cite[\S 1.2]{D-1}). Here we use that $2\Delta_{\dbar}$ and $\Delta$ have the same expression at the origin. This completes the proof of Theorem \ref{th:2}.

\begin{remark}\label{rem:notationdiff}
As mentioned, this expression for $B_2$ has been independently derived by Wang and by Xu.  Note that our notation for the Laplacians $\Delta$ and $\Delta_{\dbar}$ differs respectively by a factor of $-2$ and by a factor of $-1$ with respect to \cite[Formula (4.6)]{Lu}, \cite[Formula (4.6)]{WangLi} and \cite[Theorem~4.2]{Xu}. Furthermore, our $\Lambda$ differs by a factor of $i$ with respect to \cite[Theorem~4.2]{Xu}.
\end{remark}

For the proof of our main result in \secref{sec:ProofTheorem} we will use the compact formula
\begin{equation}\label{eq:trB2}
\begin{split}
\tr B_2 &= - \frac{1}{4}\Delta(i\Lambda \tr F) -\frac{1}{4}\tr\Lambda^2(F + \frac{1}{2}(\tr R)\Id)^2\\
& - \frac{r}{6}\Delta(S) + \frac{r}{48}\tr \Lambda^2 R^2,
\end{split}
\end{equation}
which follows easily from the identities
$$
\tr \gamma\wedge\beta\wedge\frac{\omega^{n-2}}{(n-2)!} = \frac{1}{2}(\Lambda^2 \tr \gamma\wedge\beta) \frac{\omega^n}{n!} = \tr(\gamma_{j,\overline{k}}\beta_{k,\overline{j}} + \Lambda\gamma\Lambda\beta)\frac{\omega^n}{n!},
$$
where $\gamma,\beta$ are arbitrary $(1,1)$ forms with values in the skew-adjoint endomorphisms of a unitary bundle. Note that \eqref{eq:trB2} recovers the second order term for the Riemann-Roch formula
$$
\int_X \tr B_k \frac{\omega^n}{n!} = \dim H^0(X,E\otimes L^{\otimes k}) = \int_X ch(L^{\otimes k}) ch(E) Td(X)
$$
by integration over $X$.

\section{Limits of balanced metrics}
\label{sec:ProofTheorem}

We now prove our main theorem and find the equations satisfied by limits of $\alpha$-balanced metrics.  Before doing so we give an account of the case of balanced metrics on manifolds which is simpler but illustrates the fundamental ideas.
\begin{lemma}\label{lem:fundamentallimitdonaldson}
  Suppose $b_{1,k}$ are real-valued functions on $X$, which converge to a limit $b_1$ pointwise as $k$ tends to infinity. Suppose moreover there is a pointwise expansion
$$ k^n + b_{1,k}k^{n-1}+ O\left(k^{n-2}\right)=p_k$$
where $p_k$ is a polynomial in $k$ that is constant (over $X$). Then $b_1=const$.
\end{lemma}
\begin{proof}
  This is trivial, for writing $p_k = k^n + a_1 k^{n-1} + \cdots$, the hypothesis becomes $b_{1,k} = a_1 + O(1/k)$ pointwise and taking the limit as $k$ tends to infinity gives the result.
\end{proof}
Donaldson's key observation is that one can apply the above to the Bergman kernel expansion applied to a convergent sequence of metrics.  In fact, letting $\rho_k(h)$ denote the Bergman kernel of a metric $h$, the asymptotic expansion holds uniformly over a compact set of metrics (see Theorem~\ref{th:Bergman}).  Thus if $h_k$ is a sequence of hermitian metrics on $L$ converging to $h$ in $C^\infty$ such that $h_k$ is balanced with respect to $k$, then we have
$$ p_k = \rho_k(h_k) = k^n + \frac{S_k}{2} k^{n-1} + O(k^{n-2}),$$
pointwise for suitable $p_k$, which by the above Lemma implies that the scalar curvature of the limit metric $S_h$ is constant\medskip.

For triples $(X,L,E)$ we have seen in \eqref{prop:balancedmetrics} that the balanced condition can be reinterpreted in terms of two Bergman functions being constant.  For this there is a generalisation of the above lemma, which has an extra feature in the special case that the leading order terms of the two are related.  Fix polynomials
\begin{equation}\label{eq:polynomials}
\begin{split}
p_k &= k^n + a_1 k^{n-1} + a_1 k^{n-2} + \cdots \\
\tilde{p}_k &= k^n + \tilde{a}_1 k^{n-1} + \tilde{a}_2 k^{n-2} + \cdots
\end{split}
\end{equation}

\begin{lemma}\label{lem:fundamentallimittriple}
Suppose $b_{i,k}$ are real-valued functions, and $B_{i,k}$ are endormophism valued functions of rank $r$ on a manifold $X$ such that
  \begin{eqnarray*}
k^n + b_{1,k}k^{n-1} + b_{2,k}k^{n-2} + O\left(k^{n-3}\right) &=& p_k\\
k^n\Id + B_{1,k}k^{n-1} + B_{2,k}k^{n-2} + O\left(k^{n-3}\right) &=& \tilde{p_k} \Id
\end{eqnarray*}
pointwise on $X$. Assume also that, pointwise,  $b_{i,k}$ has a limit $b_i$ and $B_{i,k}$ a limit $B_i$ as $k$ tends to infinity.
\begin{enumerate}
\item The limits satisfy
$$ B_1=const \cdot\Id\quad\text{and}\quad b_1=const.$$
\item Suppose furthermore that for real constants $\chi_1$, $\chi_2$ we have
\begin{equation}
\chi_1 \tr B_{1,k} = \chi_2  b_{1,k}\quad \text{ for all } k.\label{eq:abstractcouplecondition}
\end{equation}
Then $\chi_1r\tilde{a}_1=\chi_2a_1$ and the limits in fact satisfy
$$B_1=const \cdot \Id \quad \text{and} \quad \chi_1 \tr(B_2) = \chi_2 b_2 + const.$$
\end{enumerate}
\end{lemma}
\begin{proof}
We have
$B_{1,k} = \tilde{a}_1\Id + \tilde{a}_2/k\Id - B_{2,k}/k + O(k^{-2})$   and so taking the limit gives $B_1$ is constant, and similarly for $b_1$.  So assume \eqref{eq:abstractcouplecondition} holds.  Then
\begin{eqnarray*}
\chi_1\left(r \tilde{a}_1 + \frac{r \tilde{a}_2}{k} - \frac{\tr B_{2,k}}{k} + O(k^{-2})\right) &=& \chi_1 \tr B_{1,k} = \chi_2 b_{1,k} \\&=& \chi_2 (a_1 + \frac{a_2}{k} - \frac{b_{2,k}}{k} + O(k^{-2}))
\end{eqnarray*}
So, the top terms are equal and taking the limit as $k$ tends to infinity yields the desired result.
\end{proof}
We apply this to the Bergman kernels appearing in the balanced condition.  To do this we need to calculate the relevant coefficients which we do now.  To start with we collect the relevant topological constants by defining
\begin{eqnarray}\label{eq:definitionoftopologicalconstants}
 \beta_i &=& \frac{(n-i)!}{r^i} \sum_{p=0}^{n+1} p \binom{n+1-p}{i}\alpha_pr^{n+1-p},\\
 \gamma_i &=& \frac{(n-i)!}{r^i} \sum_{p=0}^{n+1} (n+1-p) \binom{n-p}{i}\alpha_pr^{n-p},\\
 \lambda' &=& \frac{\hat S}{2} - r\lambda \frac{\beta_1}{\beta_0},\\
 \kappa &=& 4r\(\frac{\gamma_2}{\gamma_0} - \frac{\beta_2}{\beta_0}\).
\end{eqnarray}

\begin{theorem}
Fix constants $\alpha_i$ so that $\beta_0$ and $\gamma_0$ are both non-zero. Let $(h_k,H_k)$ be a sequence of metrics on $(X,L,E)$ which converge in $C^\infty$ to $(h,H)$ as $k$ tends to infinity. Suppose furthermore that $(h_k,H_k)$ is $\alpha$-balanced with respect to $k$ for $k$ sufficiently large. Then
\begin{enumerate}
  \item if $\beta_1(\beta_0 + r\gamma_0) \neq 0$ the limit is a solution of
  \begin{equation}\label{eq:decoupledalpha}
i \Lambda F_H = \lambda \Id, \qquad S_\omega = \hat S.
\end{equation}
\item if $\beta_1(\beta_0 + r\gamma_0) = 0$ then the limit is a solution of
\begin{equation}\label{eq:coupledalpha1}
i\(\Lambda F_H - \frac{1}{r}\Lambda\tr F_H\Id\) + \(\frac{1}{2}S_{\omega} - \frac{\beta_1}{\beta_0}i\Lambda \tr F_H\)\Id = \lambda' \Id,
\end{equation}
\begin{equation}\label{eq:coupledalpha2}
(\Delta - 4\lambda')i\Lambda \tr F_H - \tr\Lambda^2(F_H^2 + F_H \wedge \tr R_\omega) - \kappa\Lambda^2 (\tr F_H)^2 = c,
\end{equation}
\end{enumerate}
where $c$ is a real constant.
\end{theorem}

Proceeding to the proof, recall that the top forms $dV_i$ used to define the balanced condition depend on $k$.  To control this, define functions $\varphi_i = \varphi_i(k)$ by
\begin{align*}
\varphi_1 \omega_k^{[n]} = dV_1 &:= \sum_{p=1}^{n+1} p\alpha_p(k\omega_k)^{p-1}\wedge(kr\omega_k + i\tr F_k)^{n+1-p},\\
\varphi_2 \omega_k^{[n]} = dV_2 &:= \sum_{p=0}^n(n+1-p)\alpha_p(k\omega_k)^p\wedge(kr\omega_k + i\tr F_k)^{n-p}.
\end{align*}

\begin{lemma}\label{lem:varphiexpansion}
We have
\begin{eqnarray*}\label{eq:lemvarphiexpansion1}
\varphi_1 &=& \sum_{j=0}^n f_jk^{n-j},\\
\varphi_2 &=& \sum_{j=0}^n g_jk^{n-j},
\end{eqnarray*}
where
\begin{eqnarray*}
  f_0 &=& \beta_0, \quad f_1 = \beta_1 \Lambda_ki\tr F_k,\quad \text{ and}\quad  f_2 =\beta_2 \Lambda^2_k(i\tr F_k)^2,\\
g_0 &=& \gamma_0, \quad g_1 = \gamma_1 \Lambda_k i\tr F_k,\quad \text{ and}\quad g_2=\gamma_2 \Lambda^2_k(i\tr F_k)^2.
\end{eqnarray*}
\end{lemma}

The proof follows from an elementary calculation using the binomial expansion. To use the Bergman kernel expansion from ~\secref{sec:Bergman}, we observe that since $(h_k,H_k)$ converges and
\begin{equation}\label{eq:betasnonzero}
\beta_0 \neq 0, \qquad \gamma_0 \neq 0,
\end{equation}
it follows from Lemma \ref{lem:varphiexpansion} that $dV_j$ are volume forms for large $k \gg 0$, so the balance condition for $(h_k,H_k)$ makes sense. Moreover,
\begin{equation}\label{eq:varphijpositive}
\varphi_j' = \frac{\varphi_j}{\int_XdV_j} >0.
\end{equation}

\begin{lemma}\label{lem:expansion}
There exists $C^\infty$ expansions
\begin{eqnarray*}
\rho_k(h_k,H_k)\varphi_1'(k) &=& \frac{k^n}{(2\pi)^n}(1 + c_1k^{-1} + c_2 k^{-2} + O(k^{-3}))\\
B_k(h_k,H_k)\varphi_2'(k) &=& \frac{k^n}{(2\pi)^n}(\Id + D_1 k^{-1} + D_2 k^{-2} + O(k^{-3}))
\end{eqnarray*}
\end{lemma}
where
\begin{eqnarray*}
  c_1 &=& \frac{S_k}{2}, \\
  c_2 &=& \frac{1}{2\beta_0}\Delta_k(f_1) - \frac{1}{16}\Lambda^2_k (\tr R_k)^2 - \frac{1}{6}\Delta_k(S_k)  + \frac{1}{48}\tr \Lambda_k^2 R_k^2, \\
 D_1 &=& i\Lambda_kF_k + \frac{S_k}{2}\Id, \\
\tr D_2 &=& \frac{r}{2\gamma_0}\Delta_k(g_1) - \frac{1}{4}\Delta_k(i\Lambda_k \tr F_k)- \frac{1}{4}\tr\Lambda^2_k(F_k + \frac{1}{2}(\tr R_k) Id)^2\\
& & - \frac{r}{6}\Delta_k S_k + \frac{r}{48}\tr \Lambda_k^2 R_k^2.
\end{eqnarray*}
\begin{proof}
If $\{t_j\}$ is a basis of $H^0(L^{\otimes k})$ used to compute $\rho_k(h_k,H_k)$ \eqref{eq:bergmanfunction1} then it is orthonormal with respect to the $L^2$ metric induced by $\omega^{[n]}$ and $h_k^k \varphi'_1$. Now, by Theorem~\ref{th:Bergman} there exists a $C^\infty$ expansion
\begin{equation}
\rho_k(h_k,H_k)\varphi'_1 = \sum_j |t_j|^2_{h_k^k \varphi_1'} = \frac{k^n}{(2\pi)^n}(\Id + \tilde{b}_1k^{-1} + \tilde{b}_2 k^{-2} + O(k^{-3}))
\end{equation}
where
\begin{equation}\label{eq:thmbB12}
\begin{split}
\tilde{b}_1 & = i\Lambda_k(\dbar\partial \log \varphi_1) + \frac{S_k}{2} = \frac{1}{2}\Delta_k \log \varphi_1 + \frac{S_k}{2},\\
\tilde{b}_2 & = - \frac{1}{4}\Delta_k(i\Lambda_k \dbar\partial\log \varphi_1) - \frac{1}{4}\Lambda^2_k(\dbar\partial \log \varphi_1 + \frac{1}{2}\tr R_k)^2\\
& - \frac{1}{6}\Delta_k S_k + \frac{1}{48}\tr \Lambda_k^2 R_k^2
\end{split}
\end{equation}
Since $(h_k,H_k)$ converge in $C^\infty$, the functions $f_j=f_j(k)$ in Lemma~\ref{lem:varphiexpansion} can be bounded in $C^m$ by a constant independent of $k$ and therefore we have a $C^\infty$ expansion
$$
\dbar\partial \log \varphi_1 = k^{-1} \beta_0^{-1}\dbar\partial(f_1) + O(k^{-2}),
$$
which gives the first part.  Similarly for the endomorphism part we obtain
$$B_k(h_k,H_k)\varphi_2' = \sum_j (s_js_j^*)_{H_k \otimes h_k^k\varphi_2'} = \frac{k^n}{(2\pi)^n}(\Id + \tilde{B}_1k^{-1} + \tilde{B}_2 k^{-2} + O(k^{-3}))$$
where
\begin{equation}
  \begin{split}
    \tilde{B}_1 & = i\Lambda_kF_k + \frac{1}{2}\Delta_k \log \varphi_2\Id + \frac{S_k}{2}\Id,\\
\tr \tilde{B}_2 & = - \frac{1}{4}\Delta_k(i\Lambda_k \tr F_k + \frac{r}{2}\Delta_k \log \varphi_2)\\
& - \frac{1}{4}\tr\Lambda^2_k(F_k + \dbar\partial \log \varphi_2\Id + \frac{1}{2}(\tr R_k)Id)^2 - \frac{r}{6}\Delta_k S_k + \frac{r}{48}\tr \Lambda_k^2 R_k^2
  \end{split}
\end{equation}
which, by an analogous expansion for $\dbar\partial \log \varphi_2$, gives the desired conclusion.
\end{proof}

\begin{proof}[Proof of Theorem]  By the balanced hypothesis and Lemma \ref{lem:expansion} we have pointwise asymptotics
\begin{equation}\label{eq:thm0.1}
(k^n + c_1k^{n-1} + c_2 k^{n-2} + O(k^{n-3}))\varphi_1^{-1} = \frac{(2\pi)^nN_k}{\int_XdV_1},
\end{equation}
\begin{equation}\label{eq:thm0.2}
(\Id k^n + D_1k^{n-1} + D_2 k^{n-2} + O(k^{n-3}))\varphi_2^{-1} = \frac{(2\pi)^nM_k}{r\int_XdV_2} \Id,
\end{equation}
where $N_k$, $M_k$ are defined in \eqref{eq:NkMk} and the $c_i,D_i$ are as in Lemma \ref{lem:expansion}.
By elementary manipulation this becomes
\begin{eqnarray}
  k^n + b_1 k^{n-1} + b_2 k^{n-2} + O(k^{n-3})&=&p_k \label{eq:elementarymanip1}\\
  k^n + B_1 k^{n-1} + B_2 k^{n-2} + O(k^{n-3})&=&\tilde p_k\Id\label{eq:elementarymanip2}
\end{eqnarray}
where $p_k$, $\tilde p_k$ are polynomials (as in \eqref{eq:polynomials}), and
$$ b_1 = c_1 -f_0^{-1}f_1, \quad b_2 = -f_0^{-1}f_1b_1 + f_0^{-1} (f_0c_2-f_2),$$
with a similar expression for the $B_i$, only $c_i,f_i$ replaced by $D_i,g_i$ respectively.  Hence we are in a position to apply Lemma \ref{lem:fundamentallimittriple}. Taking $\chi_1$, $\chi_2 \in \RR$ yields
$$
\chi_1\tr B_1 - \chi_2 b_1 = \frac{\chi_1r-\chi_2}{2}S_k + \frac{\beta_1(\chi_1\beta_0 + \chi_2\gamma_0)}{\gamma_0\beta_0}i\tr \Lambda_k F_k,
$$
where we have used the identity $\gamma_0-r\gamma_1=\beta_1$. Taking $\chi_1 = 1$, $\chi_2 = r$, the coupling condition \eqref{eq:abstractcouplecondition} depends on the vanishing of $\beta_1(\beta_0+r\gamma_0)$, resulting in two cases:  \medskip

{\bf Case 1: } $\beta_1(\beta_0+r\gamma_0)\neq 0$. Then we can apply the first part of Lemma~\ref{lem:fundamentallimittriple}, we see that the limit metrics $(h,H)$ satisfy the equations $b_1=const$ and $B_1=const\cdot Id$ which become
\begin{eqnarray*}
\beta_0 \frac{S_{\omega}}{2} - \beta_1 \tr \Lambda iF_{H} &=& const\\
\gamma_0 \left(i\Lambda F_H + \frac{S_{\omega}}{2} \Id \right) - \gamma_1 \tr(i\Lambda F_H) \Id &=& const\cdot \Id
\end{eqnarray*}
Taking the trace of the second equation, and using the condition $\beta_1(\beta_0+r\gamma_0)\neq 0$ it is easy to verify this system is equivalent to
  \begin{equation}
i \Lambda F_H = \lambda \Id, \qquad S_\omega = \hat{S}\label{eq:uncoupled}
\end{equation}
as required. \medskip

{\bf Case  2:} $\beta_1(\beta_0+r\gamma_0) =0$.  Then we can apply the second part of Lemma~\ref{lem:fundamentallimittriple} to deduce the limits satisfy the equations precisely as in the statement of the theorem.

\end{proof}

\section{Infinite dimensional picture}\label{sec:infinite}
In the spirit of~\cite{D1} and~\cite{W4}, in this section we give a different interpretation of the balanced condition~\eqref{eq:balancedcondition} defined in \secref{subsec:balancedemb}, in terms of a \emph{two-step} symplectic quotient on an infinite dimensional manifold. The main novelty of our interpretation is that the two symplectic quotients are not performed with respect to the same symplectic form, and so they are not a \emph{double quotient} in the sense of~\cite{D1}. For the first quotient, we consider a weighted sum of the symplectic structures considered in~\cite{D1} and~\cite{W4} and a \emph{extended gauge group} as the group of symmetries, as defined in~\cite{GF}. The second quotient is then taken with respect to a finite dimensional special unitary group and a small perturbation of the reduced symplectic form.
\subsection{Hamiltonian action of the extended gauge group}\label{subsec:preinfinite}
Let $(E,H)$ be a smooth hermitian vector bundle over a symplectic manifold $(X,\omega)$ endowed with a prequantization $(L,h,\nabla^L)$, with $\omega = i F_{\nabla^L}$. In this section we calculate a moment map for the action of a \emph{extended gauge group}, canonically attached to the previous data, on the space of sections of $E\otimes L^{\otimes k}$. We will find a convenient expression for the moment map evaluated on a holomorphic section that will be used in our two-step reduction in~\secref{subsec:twostepred}. For this, we will generalise previous moment map calculations in~\cite{D1} and~\cite{W4}.

First recall some general facts about extended gauge groups, following~\cite{GF}. Let $\cH$ be the group of Hamiltonian symplectomorphisms of $(X,\omega)$. Then, the extended gauge group $\cX$ of $(E,H)$ over $(X,\omega)$ is the group of hermitian automorphisms of the bundle $E$ which project onto $\cH$. It defines a non trivial extension
$$
1 \to \cG \to \cX \to \cH \to 1,
$$
of $\cH$ by the gauge group $\cG$ of $(E,H)$. To deal with the $L^{\otimes k}$ twist of the bundle we consider the fibre product
\begin{equation}\label{eq:cXL}
\cXL \defeq \cX \times_{\cH} \cHL,
\end{equation}
where $\cHL$ denotes the group of hermitian automorphisms of $L$ which preserve the connection. The infinite dimensional Lie group~\eqref{eq:cXL} is the group of symmetries of our moment map construction.
To describe the Lie algebra of~\eqref{eq:cXL}, note that $\cXL$ fits in the short exact sequence
$$
1 \to \cG \times \RR \to \cXL \to \cH \to 1.
$$
Hence, using a unitary connection $A$ on $E$, any $\zeta\in\LieXL$ can be uniquely written as
\begin{equation}\label{eq:zeta1}
\zeta = (\theta_A \zeta + \theta_A^\perp \check \zeta,if + \theta_{\nabla^{L}}^\perp \check \zeta),
\end{equation}
where $\check \zeta \in \LieH$ denotes the Hamiltonian vector field of $f \in C^\infty(X)$ covered by $\zeta$ and $\theta_A \zeta$ is the skew adjoint endomorphism of $E$ given by the vertical part of $\zeta$. The symbol $\theta^\perp$ denotes horizontal lift of vector fields. Let $k \in \ZZ$ and consider the unitary connection on $E\otimes L^{\otimes k}$ given by
\begin{equation}\label{eq:nabla}
\nabla = \nabla_A \otimes \Id + \Id \otimes \nabla^{L^{\otimes k}}.
\end{equation}
Then, the group $\cXL$ acts naturally on $E\otimes L^{\otimes k}$ and the vector field induced by any $\zeta\in\LieXL$, that we denote by the same symbol, can be decomposed as
\begin{equation}\label{eq:zeta0}
\zeta = \theta_\nabla \zeta + \theta_\nabla^\perp \check \zeta
\end{equation}
with respect to $\nabla$, where $\theta_\nabla \zeta$ admits the following explicit description
\begin{equation}\label{eq:zeta2}
\theta_\nabla \zeta = \theta_A \zeta \otimes 1 + \Id_E \otimes ikf.
\end{equation}
We will now calculate a moment map for the natural left action of $\cXL$ on the space of smooth sections $\Gamma(E\otimes L^{\otimes k})$ of $E\otimes L^{\otimes k}$. For this, consider the $1$-form $\sigma$ on $\Gamma(E\otimes L^{\otimes k})$ given by
$$
\sigma(\dot s) = -\frac{1}{2}\operatorname{Re}\int_X (i \dot s,s)\omega^{[n]}
$$
where $(\dot s_1,s)$ denotes the hermitian product on the fibres of $E \otimes L^{\otimes k}$ and $\omega^{[n]}=\omega^n/n!$.   Define the exact $2$-form
$$
\Omega = d\sigma.
$$
\begin{lemma}
\label{lem:Omega}
The $1$-form $\sigma$ is invariant under the action of $\cXL$ and
\begin{equation}\label{eq:Omegainfinite}
\Omega(\dot s_1,\dot s_2) = \operatorname{Re}\int_X (i \dot s_1,\dot s_2)\omega^{[n]}.
\end{equation}
\end{lemma}
\begin{proof}
The invariance follows trivially from the definition of $\cXL$. To compute~\eqref{eq:Omegainfinite} we can consider constant vector fields $\dot s_j \in \Gamma(E\otimes L^{\otimes k})$, with $j = 1,2$. Hence, $[\dot s_1,\dot s_2] = 0$ and
\begin{align*}
\Omega(\dot s_1,\dot s_2) & = \dot s_1(\sigma(\dot s_2)) - \dot s_1(\sigma(\dot s_2))\\
& = -\frac{1}{2}\operatorname{Re}\int_X (i \dot s_2,\dot s_1)\omega^{[n]} + \frac{1}{2}\operatorname{Re}\int_X (i \dot s_1,\dot s_2)\omega^{[n]}\\
& = \operatorname{Re}\int_X (i \dot s_1,\dot s_2)\omega^{[n]}.
\end{align*}
\end{proof}
Formula~\eqref{eq:Omegainfinite} shows that $\Omega$ defines symplectic form which has been previously considered in~\cite{D1,W4}. As an elementary consequence we obtain the existence of the desired moment map.
\begin{proposition}
\label{prop:mmapsections}
The $\cXL$-action on $(\Gamma(E\otimes L^{\otimes k}),\Omega)$ is Hamiltonian, with moment map
$$
\langle\mu,\zeta\rangle = - Y_{\zeta} \lrcorner \sigma
$$
where $Y_\zeta$ denotes the infinitesimal action of $\zeta$.
\end{proposition}
To compute an explicit expression for $\mu$, we choose a unitary connection $A$ on $(E,H)$ and consider the unitary connection $\nabla$ on $E\otimes L^{\otimes k}$ given by~\eqref{eq:nabla}. Then, using~\eqref{eq:zeta0}, a straightforward computation shows that the infinitesimal action of $\zeta \in \LieXL$ on $s \in \Gamma(E\otimes L^{\otimes k})$ is given by
$$
Y_{\zeta|s} = - \check{\zeta} \lrcorner \nabla s + \theta_\nabla \zeta \cdot s,
$$
which combined with~\eqref{eq:zeta2} and the equalities
$$
\check{\zeta} \lrcorner \nabla s \; \omega^{[n]} = - df \wedge \nabla s \wedge \omega^{[n-1]},
$$
$$
d(\nabla s,s) = d(d(|s|^2) - (s,\nabla s)) = - d(s,\nabla s) = - \overline{d(\nabla s,s)},
$$
leads to the desired expression
\begin{equation}\label{eq:muexplicit}
\begin{split}
\mu & = \frac{1}{2}\operatorname{Re}\int_X (-i\check \zeta \lrcorner \nabla s + i\theta_\nabla \zeta \cdot s,s)\omega^{[n]}\\
& = \frac{i}{2}\int_X (\theta_\nabla \zeta \cdot s,s)\omega^{[n]} - \frac{i}{2}\int_X f d(\nabla s,s)\wedge\omega^{[n-1]}\\
& = \frac{i}{2}\int_X \tr \theta_A\zeta \otimes 1 \cdot (ss^*) \omega^{[n]} + \frac{i}{2}\int_X f\(ik|s|^2\omega^{[n]} - d(\nabla s,s)\wedge\omega^{[n-1]}\).
\end{split}
\end{equation}
Suppose now that $(X,\omega)$ admits a compatible complex structure such that
$$
F_A^{0,2}= 0,
$$
so $E$ inherits a structure of holomorphic vector bundle over $X$. The moment map~\eqref{eq:muexplicit} admits a more convenient expression when $s \in \Gamma(E\otimes L^{\otimes k})$ is a holomorphic section with respect to the induced holomorphic structure. To see this, note that (cf.~\cite[Lemma~9]{D1})
$$
d(\nabla s,s) = \nabla s \wedge \nabla \overline{s} + (F_\nabla s,s),
$$
where $\overline{s}$ is the section of the dual bundle $E^*\otimes L^{-k}$ defined using the standard anti-linear isomorphism between $E\otimes L^{\otimes k}$ and $E^*\otimes L^{-k}$, and
$$
\Delta |s|^2 = 2 i \Lambda \dbar \partial |s|^2 = 2i(\Lambda(\nabla s \wedge \nabla \overline{s}) + (\Lambda F_\nabla s,s)),
$$
so \eqref{eq:muexplicit} can be rewritten as
\begin{equation}\label{eq:muexplicithol}
\begin{split}
\mu & = \frac{i}{2}\int_X \tr \theta_A\zeta \otimes 1 \cdot (ss^*) \omega^{[n]} - \frac{1}{4}\int_X f\(\Delta(|s|^2) + 2k|s|^2\)\omega^{[n]}.
\end{split}
\end{equation}
\subsection{Two-step reduction}\label{subsec:twostepred}
Let $\cJ$ denote the space of complex structures on $X$ compatible with $\omega$. Let $\cA$ be the space of unitary connections on $(E,H)$ and $\cP \subset \cJ \times \cA$ the subspace cut out by the compatibility condition
$$
F_A^{0,2_J} = 0.
$$
Note that any point in $\cP$ defines a structure of polarised manifold on $(X,L)$ and a structure of holomorphic vector bundle on $E$ over $X$. Given $k,N,M$ positive integers, consider the space of holomorphic tuples
\begin{equation}\label{eq:T}
\cT \subset \cP \times \Gamma(L^{\otimes k})^N \times \Gamma(E \otimes L^{\otimes k})^M,
\end{equation}
given by those $(J,A,t,s)$ such that
$$
t = (t_1, \ldots, t_N) \qquad \textrm{and} \qquad s = (s_1, \ldots, s_M)
$$
are basis of the spaces of holomorphic sections on $L^{\otimes k}$ and $E \otimes L^{\otimes k}$, with respect to the Dolbeault operators induced by $J$, $\nabla^L$ and $A$.
To do the first reduction we endow $\cT$ with a K\"ahler structure. For this, recall first from~\cite[\S 2]{GF} that $\cP$ has a natural complex structure (that we will not use explicitly). The complex structure on the last two factors in~\eqref{eq:T} is given simply by multiplication by $i$ and the complex structure on $\cT$ is then induced by the product structure. Following~\cite{D1}, we consider the holomorphic projection
\begin{equation}\label{eq:Tembedded}
\cT \to \Gamma(L^{\otimes k})^N \times \Gamma(E \otimes L^{\otimes k})^M
\end{equation}
and pull-back the product K\"ahler structure on the target space induced on each factor by~\eqref{eq:Omegainfinite}, weighting the second factor by a positive real constant $\epsilon > 0$. Arguing as in~\cite{D1}, for sufficiently large $k$ one can show that~\eqref{eq:T} is a holomorphic embedding (away from singularities) and that the pull-back $2$-form is indeed a K\"ahler structure, explicitly given by
\begin{equation}\label{eq:OmegaT}
\Omega_\cT(\dot \tau_1,\dot \tau_2) = \sum_j\operatorname{Re}\int_X (i \dot t_{1,j},\dot t_{2,j})\omega^{[n]} + \epsilon \sum_j \operatorname{Re}\int_X (i \dot s_{1,j},\dot s_{2,j})\omega^{[n]}
\end{equation}
where $\dot \tau_l = (\dot J_l,\dot A_l,\dot t_l,\dot s_l)$ with $l = 1,2$.
The left action of $\cXL$ on each of the factors in~\eqref{eq:T} induces a well-defined action on $\cT$ which preserves the K\"ahler structure and, as a direct consequence of the moment map computation in~\secref{subsec:preinfinite}, it follows that the $\cXL$-action is Hamiltonian. In order to apply this fact, we need to twist~\eqref{eq:muexplicithol} by an element in the center of the Lie algebra of gauge group $\LieG$, given by $i\nu\Id$ for $\nu \in \RR$.
\begin{lemma}\label{lem:mmap}
Given real constants $c, \nu \in \RR$, the following expression defines a moment map for the $\cXL$-action on $\cT$
\begin{equation}\label{eq:mulambda}
\begin{split}
\mu_{\nu} & = \frac{i\epsilon}{2}\int_X\tr \theta_A \zeta \otimes 1 \cdot \(\sum_js_js_j^* - \nu \Id\)\omega^{[n]}\\
& - \frac{1}{4}\int_X f \(\Delta\(\sum_j|t_j|^2+\epsilon\sum_j|s_j|^2\) + 2k\sum_j|t_j|^2 + 2k\epsilon\sum_j|s_j|^2 - c\)\omega^{[n]}\\
& + \frac{i\nu\epsilon}{2}\int_X f\tr\Lambda F_A\omega^{[n]}.
\end{split}
\end{equation}
\end{lemma}
\begin{proof}
Note first that $\mu_0$ is a moment map for the $\cXL$-action by~\eqref{eq:muexplicithol}, since the map $\zeta \to \int_X f\omega^{[n]}$ defines a character of $\LieXL$ constant over $\cT$. We prove now that
$$
\mu' := \frac{i\nu\epsilon}{2}\int_X \tr (\theta_\nabla\zeta - f \Lambda F_\nabla)\omega^{[n]}
$$
is constant on $\cA$, where $\nabla$ is given by~\eqref{eq:nabla}. Using~\eqref{eq:zeta2} and
\begin{equation}\label{eq:Fnabla}
\begin{split}
F_\nabla = F_A \otimes 1 - ik \Id_E \otimes \omega
\end{split}
\end{equation}
we calculate
\begin{align*}
\frac{d}{dt}_{|t=0}\mu'(A + t\dot A) &= \frac{i\nu\epsilon}{2}\int_X \tr (\check \zeta \lrcorner \dot A \otimes 1 - f \Lambda d_A \dot A \otimes 1)\omega^{[n]}& \\
&= \frac{i\nu\epsilon}{2}\int_X -df \wedge \tr\dot A \wedge\omega^{[n-1]} - f d \tr\dot A \wedge\omega^{[n-1]} & = 0
\end{align*}
for any $A \in \cA$ and $\dot A$ in the tangent space $T_A \cA$, identified with the space of $\End(E,H)$-valued $1$-forms on $X$, where we have used the equalities
$$
\frac{d}{dt}_{|t=0}\dot \theta_{\nabla_t} = \dot A \otimes 1, \qquad  \check{\zeta} \lrcorner \dot A \; \omega^{[n]} = - df \wedge \dot A \wedge \omega^{[n-1]}.
$$
Finally, the statement follows rewriting $\mu'$ using \eqref{eq:Fnabla} and the equality
$$
\mu_{\nu} = \mu_0 - \mu' -\(\nu\epsilon kr(n+1)/2 - c/4\)\int_X f\omega^{[n]}.
$$
\end{proof}

We claim that the zero locus of $\mu_\nu$ is given by the solutions of the system
\begin{equation}\label{eq:munuzeros}
\begin{split}
\sum_js_js_j^* & = \nu \Id\\
\Delta\(\sum_j|t_j|^2\) + 2k\sum_j|t_j|^2 & = \epsilon 2i\nu\tr\Lambda F_A + c'
\end{split}
\end{equation}
for a constant $c' \in \RR$ (compare with Proposition~\ref{prop:balancedmetrics}). To see this, we first evaluate~\eqref{eq:mulambda} on vertical vector fields, obtaining the first equation in~\eqref{eq:munuzeros}. Now, taking trace of this equation and evaluating~\eqref{eq:mulambda} on horizontal vector fields with respect to $A$ we obtain the second equation in~\eqref{eq:munuzeros}. To discuss the existence of solutions of \eqref{eq:munuzeros}, we introduce the notion of complexified orbit. Recall from \cite[\S 3]{AGG} that in $\cP$ there is a well defined notion of complexified orbit for the extended gauge group $\cX$ such that its points correspond to pairs of Hermitian metrics on $L$ and $E$, up to the action of $\cX$. Similarly, one can define an appropriate notion of complexified orbit for the $\cXL$-action on $\cT$ by defining and equivalence relation: $\tau \sim \tau'$ if there exists automorphisms of the complex vector bundles $E$ and $L^k$ (not necessarily preserving $h^k$, $H$ and $\nabla^{L^k}$) which project to the same diffeomorphisms on $X$ and such that take $\tau$ to $\tau'$ (cf. \cite[\S 2.1]{D1}). Again, up to the action of $\cXL$, it can be checked that points in the complexified orbit correspond to pairs of hermitian metrics on $E$ and $L$.

\begin{lemma}\label{lem:existencezeros}
Given a complexified orbit in $\cT$, it contains a solution of \eqref{eq:munuzeros} for sufficiently small $\epsilon$.
\end{lemma}
\begin{proof}
Fixing $\tau \in \cT$ and considering as unknowns a pair of metrics $(h,H)$, \eqref{eq:munuzeros} is a system in separated variables where the first equation can be easily solved by a change of metric on $E$. Moreover, for $\epsilon = 0$ it reduces  to the system
\begin{equation}\label{eq:munuepsiloncero}
\begin{split}
\sum_js_js_j^* & = \nu \Id\\
\sum_j|t_j|^2 & = c''
\end{split}
\end{equation}
for a constant $c'' \in \RR$, which admits a unique solution (that without lose of generality we assume to be $(h,H)$). Since the linearisation of the scalar equation in \eqref{eq:munuepsiloncero} when $\epsilon = 0$ is given by
$$
c''\Delta f + c''2kf = 0
$$
for a conformal change $h' = e^f h$, the statement follows from an standard implicit function theorem argument.
\end{proof}

For the second reduction, we consider an exact $2$-form on the symplectic quotient
\begin{equation}\label{eq:quotient1}
\cT//\cXL = \mu_\nu^{-1}(0)/\cXL
\end{equation}
which has a finite asymptotic expansion in the parameter $k$ and whose leading order term corresponds to~\eqref{eq:OmegaT}. For this, given parameters $\alpha = (\alpha_0,\ldots,\alpha_{n+1}) \in \RR^{n+2}$ and positive real numbers $q_1,q_2 > 0$ we define a $1$-form on $\cT$ given by
\begin{equation}\label{eq:sigmakl}
\sigma_k(\dot \tau) = - q_1\operatorname{Re}\int_X \frac{\sum_j\(i\dot t_j,t_j\)}{\sum_j|t_j|^2}dV_1(\tau) - q_2\operatorname{Re}\int_X \sum_j\(i\dot s_j,s_j\)dV_2(\tau)
\end{equation}
for $\dot \tau = (\dot J,\dot A,\dot t,\dot s) \in T_\tau \cT$ and $\tau = (J,A,t,s)$,where
\begin{align*}
dV_1(\tau)&= \sum_{p=1}^{n+1} p\alpha_p\omega_1^{p-1}\wedge\omega_2^{n+1-p},\\
dV_2(\tau) &= \sum_{p=0}^n(n+1-p)\alpha_p\omega_1^p\wedge\omega_2^{n-p},
\end{align*}
and
$$
\omega_1 = k\omega + \partial \dbar_J \log\(\sum_j|t_j|^2\), \qquad \omega_2 = kr\omega + i\tr F_A.
$$
Then, it is easy to verify that $\sigma_k$ is $\cXL$-invariant, and so the exact $2$-form
\begin{equation}\label{eq:Omegakl}
\Omega_k = d\sigma_k
\end{equation}
induces a well defined $2$-form on $\cT//\cXL$.

\begin{remark}
Relying on the asymptotic expansion of the density of states $\sum_j|t_j|^2$ in powers of $k$ (see \S\ref{sec:Bergman}), for suitable values of $q_1$ and $q_2$ (which may depend on $k$) we have an expansion
$$
\Omega_k = \Omega_\cT + O(k^{-1}).
$$
Hence, one may expect non-degeneracy of~\eqref{eq:Omegakl} for large values of $k$. The previous expansion shall be considered only at a formal level.
\end{remark}
\begin{remark}
The $2$-form $\Omega_k$ is of type $(1,1)$ only up to order $O(k^{-1})$. The failure of the compatibility with the complex structure on $\cT$ can be checked considering mixed derivatives involving the $A-z$ and $A-s$ directions. To illustrate this, suppose $N = 0$, $M=1$ and $q_2 = 1$ and compute (cf. Lemma~\ref{lem:Omega})
\begin{align*}
d\sigma_k(\dot \tau_1,\dot \tau_2) & = \operatorname{Re}\int_X (i \dot s_1,\dot s_2)dV_2\\
& + \operatorname{Re}\int_X (i\dot s_1,s)\partial_A(dV_2)(\dot A_2) - \operatorname{Re} \int_X (i\dot s_2,s)\partial_A(dV_2)(\dot A_1)\\
\end{align*}
where $\dot \tau_j$ are assumed to be of the form $(0,\dot A_j,0,\dot s_j)$. Using that the complex structure on $\cT$ at $(J,A,s)$ acts on $\dot A_j$ as $\dot A \to \dot A(-J \cdot)$ it can be readily checked that $\Omega_k$ is not of type $(1,1)$.
\end{remark}

Consider the product of special unitary groups $SU(N)\times SU(M)$ acting on $\cT$ on the left, via its action on $\Gamma(L^{\otimes k})^N \times \Gamma(E \otimes L^l)^M$. Note that this action commutes with the $\cXL$-action on $\cT$ and preserves the zero locus of $\mu_\nu$, so induces a well-defined action on~\eqref{eq:quotient1}. Moreover, it clearly preserves \eqref{eq:sigmakl} and so it induces a Hamiltonian action on $(\cT//\cXL,\Omega_k)$ (cf. Proposition~\ref{prop:mmapsections}). Using the standard identification of $\mathfrak{su}(N)\times\mathfrak{su}(M)$ with its dual, the following result is straightforward.
\begin{lemma}\label{lem:mmapSU}
The $SU(N)\times SU(M)$ action on $(\cT//\cXL,\Omega_k)$ is Hamiltonian, with moment map $\mu_{SU} = \mu_{SU(N)} \times \mu_{SU(M)}$ given by
\begin{equation}\label{eq:mmapSU}
\begin{split}
\mu_{SU(N)} & = iq_1\(\int_X \frac{\langle t_m,t_l\rangle}{\sum_j|t_j|^2} dV_1 - \frac{\int_XdV_1}{N}\Id\),\\
\mu_{SU(M)} & = iq_2\(\int_X \langle s_m,s_l\rangle dV_2 - \frac{r\nu\int_XdV_2}{M}\Id\).
\end{split}
\end{equation}
\end{lemma}

By formulae \eqref{eq:mmapSU} and \eqref{eq:mulambda}, it follows that if $\tau \in \cT$ lies in the intersection
$$
\mu_\nu^{-1}(0) \cap \mu_{SU}^{-1}(0)
$$
then the pair $(J,A) \in \cP$ defines an $\alpha$-balanced triple $(X,L,E)$ with respect to $k$, with balanced embedding determined by the basis $t$ and $s$. To state a converse, which may be compared with \cite[Proposition 11]{D1}, notice that the action of $SU(N)\times SU(M)$ on $\cT$ extends to a holomorphic action of $SL(N,\CC)\times SL(M,\CC)$ and so there is a natural notion of complexified orbit for the $\cXL \times SU(N)\times SU(M)$-action defined in the obvious way.

\begin{proposition}
A point in $\tau \in \cT_0$ defines an holomorphic structure of $\alpha$-balanced triple $(X,L,E)$ if and only if its complexified orbit for the $\cXL \times SU(N)\times SU(M)$-action contains a point in
$$
\mu_\nu^{-1}(0) \cap \mu_{SU}^{-1}(0),
$$
or equivalently, if and only if the complexified orbit is represented by a point in the two-step quotient
$$
\mu_\nu^{-1}(0) \cap \mu_{SU}^{-1}(0)/\cXL \times SU(N)\times SU(M).
$$
\end{proposition}

The proof follows from the discussion above and is left to the reader.

\appendix

\section{remaining calculations}\label{sec:appendix}

In this Appendix we record the details of the remaining calculations for the computation of the $B_2$ term in \secref{sec:Bergman}.

\begin{align*}
(D_\theta \cdot D_y)^2(c_2) & = \frac{2}{3} \partial^2_{\theta^j,\theta^k}(\tr\partial_{y^j}(\Psi^{-1}\partial_{y^k}\Psi)(x,z))\\
& = \frac{2}{3}\partial_{\theta^j}(\tr(\partial_{z^l}\partial_{y^j}(\Psi^{-1}\partial_{y^k}\Psi))(x,z)(\partial_{\theta^k}z^l))\\
& = \frac{2}{3}\tr \partial_{z^m}(\partial_{z^l}\partial_{y^j}(\Psi^{-1}\partial_{y^k}\Psi))(\partial_{\theta^j}z^m)(\partial_{\theta^k}z^l)\\
& = \frac{2}{3}\tr(\partial_{z^j}\partial_{z^k}\partial_{y^j}(\Psi^{-1}\partial_{y^k}\Psi))\\
& = \frac{2}{3}\tr \partial_{z^j}\partial_{y^j}(\Psi_{k,m}\Psi^{-1}_{m,l}\partial_{z^k}(\Psi^{-1}\partial_{y^l}\Psi))\\
& = \frac{2}{3}\tr \partial_{z^j}\partial_{y^j}(\Psi^{-1}_{k,l}\partial_{z^k}(\Psi^{-1}\partial_{y^l}\Psi))\\
& + \frac{2}{3} (\partial_{z^j}\partial_{y^j}\Psi)_{k,m}\tr \partial_{z^k}(\Psi^{-1}\partial_{y^m}\Psi)\\
& = \frac{1}{3}\Delta(S) + \frac{2}{3}(R_{j,\overline{j}})_{k,m}\tr R_{m,\overline{k}}\\
& = \frac{1}{3}\Delta(S) + \frac{2}{3}|\tr R|^2,
\end{align*}
\begin{align*}
(D_\theta \cdot D_y)^2(c_3) & = \frac{1}{4}\partial^2_{\theta^j,\theta^k}(\tr\Psi^{-1}(\partial_{y^j}\Psi)\tr\Psi^{-1}(\partial_{y^k}\Psi)(x,z))\\
& = \frac{1}{4} (\tr\partial_{\theta^j}(\Psi^{-1}\partial_{y^j}\Psi))(\tr\partial_{\theta^k}(\Psi^{-1}\partial_{y^k}\Psi))\\
& + \frac{1}{4} (\tr\partial_{\theta^k}(\Psi^{-1}\partial_{y^j}\Psi))(\tr\partial_{\theta^j}(\Psi^{-1}\partial_{y^k}\Psi))\\
& = \frac{1}{4}\tr R_{j,\overline{j}}\tr R_{k,\overline{k}}(0) + \frac{1}{4}\tr R_{j,\overline{k}}\tr R_{k,\overline{j}}\\
& = \frac{1}{4}S^2 + \frac{1}{4}|\tr R|^2,
\end{align*}
\begin{align*}
(D_\theta \cdot D_y)^2(c_4) & = -\frac{1}{12} \partial^2_{\theta^j,\theta^k}\tr(\Psi^{-1}(\partial_{y^j}\Psi)\Psi^{-1}(\partial_{y^k}\Psi)(x,z))\\
& = - \frac{1}{12} \tr(\partial_{\theta^j}(\Psi^{-1}\partial_{y^j}\Psi))(\partial_{\theta^k}(\Psi^{-1}\partial_{y^k}\Psi))\\
& - \frac{1}{12} \tr(\partial_{\theta^k}(\Psi^{-1}\partial_{y^j}\Psi))(\partial_{\theta^j}(\Psi^{-1}\partial_{y^k}\Psi))\\
& = - \frac{1}{12}\tr R_{j,\overline{j}}R_{k,\overline{k}} - \frac{1}{12}\tr R_{j,\overline{k}}R_{k,\overline{j}}\\
& = - \frac{1}{12}|\Lambda R|^2 - \frac{1}{12}|R|^2\\
& = - \frac{1}{12}|\tr R|^2 - \frac{1}{12}|R|^2,
\end{align*}
\begin{align*}
(D_\theta \cdot D_y)^2(d_1) & = 2 \partial^2_{\theta^j,\theta^k}((\partial_{z^l}(G^{-1}\partial_{y^j}G))(x,z)\partial_{y^k}z^l)\\
& = 2\tr(\partial_{z^l}(G^{-1}\partial_{y^j}G))(0)\partial^2_{\theta^k,y^k}(\partial_{\theta^j}z^l)\\
& = - 2 F_{j,\overline{l}}\partial^2_{\theta^k,y^k}((\partial_z\theta)^{-1}_{l,j})\\
& = F_{j,\overline{l}}(\partial_{\theta^k}(\Psi^{-1}\partial_{y^k}\Psi(y,z)))_{l,j}\\
& = - F_{j,\overline{l}}(R_{k,\overline{k}})_{l,j},
\end{align*}
\begin{align*}
(D_\theta \cdot D_y)^2(d_2) & = \partial^2_{\theta^j,\theta^k}(\partial_{y^j}(G^{-1}\partial_{y^k}G)(x,z))\\
& = \partial_{\theta^j}((\partial_{z^l}\partial_{y^j}(G^{-1}\partial_{y^k}G))(x,z)(\partial_{\theta^k}z^l))\\
& = \partial_{z^m}(\partial_{z^l}\partial_{y^j}(G^{-1}\partial_{y^k}G))(\partial_{\theta^j}z^m)(\partial_{\theta^k}z^l)\\
& = (\partial_{z^j}\partial_{z^k}\partial_{y^j}(G^{-1}\partial_{y^k}G))\\
& = \partial_{z^j}\partial_{y^j}(\Psi_{k,m}\Psi^{-1}_{m,l}\partial_{z^k}(G^{-1}\partial_{y^l}G))\\
& = \partial_{z^j}\partial_{y^j}(\Psi^{-1}_{k,l}\partial_{z^k}(G^{-1}\partial_{y^l}G))\\
& + (\partial_{z^j}\partial_{y^j}\Psi)_{k,m} \partial_{z^k}(G^{-1}\partial_{y^m}G)\\
& = \frac{1}{2}\Delta(i\Lambda F) + (R_{j,\overline{j}})_{k,m}F_{m,\overline{k}},
\end{align*}
\begin{align*}
(D_\theta \cdot D_y)^2(d_3) & = \partial^2_{\theta^j,\theta^k}(G^{-1}(\partial_{y^j}G)G^{-1}(\partial_{y^k}G)(x,z))\\
& = (\partial_{\theta^j}(G^{-1}\partial_{y^j}G))(\partial_{\theta^k}(G^{-1}\partial_{y^k}G))\\
& + (\partial_{\theta^k}(G^{-1}\partial_{y^j}G))(\partial_{\theta^j}(G^{-1}\partial_{y^k}G))\\
& = F_{j,\overline{j}}F_{k,\overline{k}}(0) + F_{j,\overline{k}}F_{k,\overline{j}}\\
& = - \Lambda F \Lambda F + F_{j,\overline{k}}F_{k,\overline{j}}.
\end{align*}

\end{document}